\documentclass [12pt]{amsart}

\usepackage[utf8]{inputenc}

\usepackage{amssymb,amsmath,amsthm}
\newtheorem{theorem}{Theorem}
\newtheorem{lemma}{Lemma}
\newtheorem{proposition}{Proposition}

\newtheorem{result}{Result}
\newtheorem{observation}{Remark}
\newtheorem{conjecture}{Conjecture}
\newtheorem{corollary}{Corollary}
\newtheoremstyle{neosn}{0.5\topsep}{0.5\topsep}{\rm}{}{\bf}{.}{ }{\thmname{#1}\thmnumber{ #2}\thmnote{ {\mdseries#3}}}
\theoremstyle{neosn}
\newtheorem{definition}{Definition}

\newcommand{\Aut}{\,\mathrm{Aut}\,}
\newcommand{\Hom}{\,\mathrm{Hom}\,}
\newcommand{\ad}{\,\mathrm{ad}\,}

\renewenvironment{proof}{\noindent \textbf{Proof.}}{$\blacksquare$}
\newcommand{\GL}{\,\mathrm{GL}\,}
\newcommand{\SL}{\,\mathrm{SL}\,}
\newcommand{\PGL}{\,\mathrm{PGL}\,}
\newcommand{\PSL}{\,\mathrm{PSL}\,}

\newcommand{\Rad}{\,\mathrm{Rad}\,}

\newcommand{\tr}{\mathrm{tr}\,}

\begin{document}

\begin{center}

{\Large {\bf The Diophantine problem in Chevalley groups}}
\end{center}
\bigskip

{\large {\bf Elena Bunina}, Bar-Ilan University}

\medskip
 
{\large {\bf Alexey Myasnikov}, Stevens Institute of Technology}

\medskip

{\large {\bf  Eugene Plotkin}, Bar-Ilan University}\footnote{Research of  Eugene Plotkin was supported by the ISF grant 1994/20}

\medskip


\bigskip

 \begin{center}

{\large{\bf Abstract}}

\end{center}

In this paper we study the Diophantine problem in Chevalley groups $G_\pi (\Phi,R)$, where $\Phi$ is an indecomposable root system of rank $> 1$, 
$R$ is an arbitrary commutative ring with~$1$.

We establish a variant of double centralizer theorem for elementary unipotents  $x_\alpha(1)$. This theorem is valid for arbitrary commutative rings with~$1$. The result is principle to show that any
 one-parametric subgroup $X_\alpha$, $\alpha \in \Phi$, is Diophantine in~$G$.
Then we prove that  the Diophantine problem in $G_\pi (\Phi,R)$ is polynomial time equivalent (more precisely, Karp equivalent) to the
Diophantine problem in~$R$. This fact gives rise to a number of model-theoretic corollaries  for specific types of rings.

\bigskip


\bigskip

\section{Introduction and State of Art}\leavevmode

Recall that the \emph{Diophantine problem} (also called the \emph{Hilbert's tenth problem} or
the \emph{generalized Hilbert's tenth problem}) in a countable algebraic structure~$\mathcal A$,
denoted $\mathcal D(\mathcal A)$, asks whether there exists an algorithm that, given a finite system
$S$ of equations in finitely many variables and coefficients in~$\mathcal A$, determines if $S$
has a solution in~$\mathcal A$ or not. In particular, if $R$ is a countable ring then $\mathcal D(R)$ asks
whether the question if a finite system of polynomial equations with coefficients
in~$R$ has a solution in~$R$ is decidable or not. It is tacitly assumed that the
ring~$R$ comes with a fixed enumeration, i,\,e., a function $\nu:\mathbb N \to R$, which
enables one to enumerate all polynomials in the ring of all non-commutative
polynomials $R\langle x_1; x_2; \dots \rangle$ (in countably many variables $x_1; x_2; \dots$), as well as all
finite systems of polynomial equations $p(x_1; \dots ; x_n) = 0$, where $p(x_1; \dots; x_n) \in 
R\langle x_1; x_2; \dots \rangle$, so one can provide them as inputs to a decision algorithm. If the
ring $R$ is commutative (it is our case) then by tradition only commutative polynomials from
$R[x_1; x_2; \dots]$ are considered. The original version of this problem was posed by
Hilbert for the ring of integers~$\mathbb Z$. This was solved in the negative in 1970 by
Matiyasevich~\cite{M52} building on the work of Davis, Putnam, and Robinson~\cite{M17}.
Subsequently, the Diophantine problem has been studied in a wide variety of
commutative rings~$R$, where it was shown to be undecidable by reducing $\mathcal D(\mathbb Z)$
to $\mathcal D(R)$. By definition the Diophantine problem in a structure~$\mathcal A$ \emph{reduces to} the Diophantine problem in a structure~$\mathcal B$, symbolically $\mathcal D(\mathcal A)\leqslant \mathcal D(\mathcal B)$, if there is an algorithm that for a given finite system of equations $S$ with coefficients in~$\mathcal A$ constructs a system of equations $S^*$ with coefficients in~$\mathcal B$ such that $S$ has a solution in~$\mathcal A$
 if and only if $S^*$ has a solution in~$\mathcal B$. So if $\mathcal D(\mathbb Z)\leqslant \mathcal D(R)$ then $\mathcal D(R)$ 
is undecidable. If the reducing algorithm is polynomial-time then the reduction is 
termed \emph{polynomial-time} (or \emph{Karp reduction}). In this paper we show that the
Diophantine problems in $G_\pi(\Phi,R)$ and $R$ are polynomial time equivalent which means,
precisely, that $\mathcal D(G_\pi(\Phi,R))$ and $\mathcal D(R)$ reduce to each other in polynomial time.
In particular they are either both decidable or both undecidable. If $R$ and hence $G_\pi(\Phi,R)$ are uncountable one needs to restrict the Diophantine problems in $R$ and $G_\pi(\Phi,R)$ to equations with coefficients from a fixed countable subset of $R$ or $G_\pi(\Phi,R)$. After a proper adjustment in definitions the Diophantine problems in $R$ and $G_\pi(\Phi,R)$ are still  polynomial time equivalent  (we will say more about this later).

A lot of research has been done on equations in commutative rings. Nevertheless, the
Diophantine problem is still open in~$\mathbb Q$ and fields $F$ which are finite algebraic
extensions of~$\mathbb Q$. Much more is known on the Diophantine problem in the rings
of algebraic integers~$\mathcal O$ of the fields~$F$. Namely, it was shown that $\mathcal D(\mathbb Z)$ reduces
to $\mathcal D(\mathcal O)$ for some algebraic number fields~$\mathcal O$, hence in such~$\mathcal O$ the Diophantine problem $\mathcal D(\mathcal O)$ is undecidable. We refer to \cite{M60}, \cite{M59}, \cite{M75} for further information on the Diophantine problem in different rings and fields of number-theoretic 
flavour. There are long-standing conjectures (see, for example, \cite{M18}, \cite{M59}) which state that
the Diophantine problems in~$\mathbb Q$; $F$, and~$\mathcal O$, as above, are all undecidable. The
following result is important for our paper. If a commutative unitary ring $R$ is
infinite and finitely generated then, in the case of a positive characteristic, $\mathcal D(R)$
is undecidable, and in the case of characteristic zero, $\mathcal D(\mathcal O)$ polynomial-time
reduces to $\mathcal D(R)$ for some ring of algebraic integers~$\mathcal O$ (Kirsten Eisentraeger's
PhD thesis (Theorem 7.1), which is available on her website, see also~\cite{M34}).

In the class of non-commutative associative unitary rings it was shown recently
by Kharlampovich and Myasnikov in~\cite{M41} that the Diophantine problem is undecidable in free associative algebras over fields and in the group algebras of a wide variety of torsion-free groups, including toral relatively hyperbolic groups, right angled Artin groups, commutative transitive groups, and the fundamental
groups of various graphs of groups. For non-associative rings it was proved that
the Diophantine problem is undecidable in free Lie algebras of rank at least
three with coefficients in an arbitrary integral domain~\cite{M40}. A general approach to the Diophantine problem in non-commutative rings (via reductions to the commutative ones) was developed in~\cite{M33}.

In another direction, coming from model theory, it was shown that the first-order theory of some classical fields is 
decidable: Tarski proved it for for complex numbers~$\mathbb C$ and reals~$\mathbb R$~\cite{M76}, and Ershov,
Ax and Kochen for p-adic numbers $\mathbb Q_p$ and $\mathbb Z_p$ (\cite{M26}, \cite{M1}, \cite{M2}). The 
statement that a given finite system of equations has a solution in~$R$ can be represented by a very particular existential formula (a positive-primitive formula) with coefficients in~$R$, so the Diophantine problem seems to be a part of the first-order theory of~$R$, but the coefficients are getting involved, and this complicates the whole picture. In fact, involvement of constants (coefficients)  makes Diophantine problems rather different from the classical
model-theoretic problems of elementary equivalence and decidability of first order theories in the standard languages of groups or rings. We will say more on this later, specifically for the linear groups and Chevalley groups.

Similar to the Diophantine problem in rings if a structure $\mathcal A$ is countable or finite then we assume that it comes equipped with an enumeration $\nu: \mathbb N\to \mathcal A$,
which enables one to enumerate all terms in the language of~$\mathcal A$ with constants
in~$\mathcal A$, hence all equations (which in this case are represented by equalities of two
terms), as well as all finite systems of equations over~$\mathcal A$. On the other hand,
if $\mathcal A$ is uncountable then, by definition, one has to consider only equations with constants from a fixed arbitrary countable (or finite) subset~$C$ of~$A$. We denote this form of the Diophantine problem by $\mathcal D_C(\mathcal A)$. This modification allows one to consider Diophantine problems over arbitrary structures in a more precise and also a more uniform way. As we will see below it may happen that the Diophantine problem $\mathcal D_C(\mathcal A)$ is decidable for one subset $C \subseteq \mathcal A$ and
undecidable for another one, even in countable structures~$\mathcal A$. Moreover, it may depend on a chosen enumeration of a countable set $C$.  It is easy to see
that for a countable (or finite) subset $C$ of~$\mathcal A$ the Diophantine problems $\mathcal D_C(\mathcal A)$
and $\mathcal D_{\langle C\rangle}(\mathcal A)$ reduce to each other, where $\langle C\rangle$ is the substructure generated by~$C$ in~$\mathcal A$. Furthermore, if $\mathcal D_C(\mathcal A)$ is decidable then $\langle C\rangle$ is computable (recursive, constructible) in the sense of Maltsev ~\cite{M50} and Rabin~\cite{M62}, so if $\langle C\rangle$ is not computable, and this may depend on the enumeration of $C$, the Diophantine problem $\mathcal D_C(\mathcal A)$ is undecidable. Therefore, from the beginning one may consider only enumerations of $C$ with computable substructure $\langle C\rangle$. 

Research on systems of equations and their decidability in groups has a very long history, it goes back to 1912 to the pioneering works of Dehn on the word and conjugacy problems in finitely presented groups. Recall that an equation in a group~$G$ is an expression of the type $w(x_1; \dots ; x_n; g_1; \dots ; g_m) = 1$, where $w$ is
a group word in variables $x_1; \dots ; x_n$ and constants $g_1; \dots ; g_m \in  G$. Currently,
there are two main approaches to the Diophantine problems in groups. In the first approach one given a fixed group $G$ tries to find a commutative unitary ring~$\mathcal A$ such that the Diophantine problem in~$\mathcal A$ algorithmically reduces to the Diophantine problem in~$G$. In this case if $\mathcal D(\mathcal A)$ is undecidable then $\mathcal D(G)$ is also undecidable. The first principle result in this vein is due to Romankov, who showed that the Diophantine problem is undecidable in any non-abelian free nilpotent group~$N$ of nilpotency class at least~$9$ (he proved that $\mathcal D(\mathbb Z)\leqslant \mathcal D(N)$ even one considers only single equations in the group~$N$)~\cite{M72}. Recently, Duchin, Liang and Shapiro showed in~\cite{M24} that $\mathcal D(\mathbb Z)\leqslant \mathcal D(N)$ for any nonabelian free nilpotent group~$N$, hence $\mathcal D(N)$ is undecidable. A far-reaching generalizations of these were obtained by Garreta, Myasnikov and Ovchinnikov in~\cite{M32} where they proved that for any finitely generated non-virtually abelian nilpotent group~$G$ there exists a ring of algebraic integers~$\mathcal O$ (depending on~$G$) interpretable by equations in~$G$, hence $\mathcal D(\mathcal O)$ is Karp reducible to $\mathcal D(G)$. Furthermore, in~\cite{M31}
they gave a general sufficient condition for the ring~$\mathcal O$ to be isomorphic to~$\mathbb Z$, so in this case the Diophantine problem in~$G$ is undecidable. Based on this, they proved that a random nilpotent group $G$ (given by a random presentation in the variety $\mathcal N_c$ of nilpotent groups of class at most~$c$, for any $c\geqslant 2$) has $\mathcal O \cong\mathbb Z$, hence the undecidable Diophantine problem. These results on nilpotent groups allow numerous applications to the Diophantine problems in non-nilpotent groups
$H$ either via suitable Diophantine nilpotent subgroups of $H$ or via suitable Diophantine nilpotent quotients of~$H$~\cite{M32}. For example, this technique allows one to show that the Diophantine problem in any finitely generated free solvable non-abelian group is undecidable.

This line of results changes drastically in the second approach, where one tries to show that the Diophantine problem in a given group~$G$ is decidable by reducing it to the Diophantine problem in a non-abelian free group~$F$ or a free monoid~$M$ (see, for example, Rips and Sela~\cite{M66}, Damani and Guirardel~\cite{M16}, Diekert and Muschol~\cite{M22}, Casals-Ruiz and Kazachkov~\cite{M13},~\cite{M12}, and Diekert and Lohrey~\cite{M21}). We refer to~\cite{M39} for further results in this area. The principal results
here are due to Makanin \cite{M46}, \cite{M47} and Razborov \cite{M63},~\cite{M64} who showed that the
Diophantine problems $\mathcal D(M)$ and $\mathcal D(F)$ are decidable and, in the case of the free group~$F$, further provided a description of the solution sets to arbitrary finite systems of equations in terms of Makanin--Razborov's diagrams. Another description of solutions sets in~$F$ in terms of NTQ systems (also termed $\omega$-residually free towers) was obtained in~\cite{M38}. NTQ systems give an effective approach to algebraic geometry and model theory of free groups. Recently, an entirely different method of solving equations in free groups, free monoids, and hyperbolic groups was developed in a series of papers \cite{M20}, \cite{M35}, \cite{M36}, \cite{M14},~\cite{M15}.

In his classical paper \cite{Maltsev} A.I.\,Maltsev studied elementary equivalence of matrix groups $\mathcal G_n(F)$ where $\mathcal G_n$ is one of the $\GL_n, \SL_n, \PGL_n, \PSL_n$, $n\geqslant 3$, and $F$ is a field. Namely, he showed that $\mathcal G_n(F)\equiv \mathcal G_m(L)$ if and only if $n=m$ and $F\equiv L$. His proof was based on two principal results. The first one states that for any integer $k\geqslant 3$ and $\mathcal G_n$ as above there is a group sentence $\Phi_{k,\mathcal G}$ such that for any~$n$, and a field~$F$, $\Phi_{k,\mathcal G}$ holds in $\mathcal G_n(F)$ if and only if $k=n$. The second one is that $F$ and $\mathcal G_n(F)$ are mutually interpretable in each other. More precisely, $\mathcal G_n(F)$ is absolutely interpretable in~$F$ (i.\,e., no use of parameters), while $F$ is interpretable in $\mathcal G_n(F)$ uniformly with respect to some definable subset of tuples of parameters (so-called \emph{regular} interpretability). This implies that the theories $Th(F)$ and $Th(\mathcal G_n(F))$ are reducible to each other in polynomial time, hence $Th(\mathcal G_n(F))$ is decidable if and only if $Th(F)$ is decidable. Later Beidar and Mikhalev introduced another general approach to elementary equivalence of classical matrix groups~\cite{BeidarMikhalev}. Their proof was based on Keisler--Shelah theorem (two structures are elementarily equivalent if and only if their ultrapowers over some non-principal ultraflters are isomorphic (see~\cite{Keisler},~\cite{Shelah}) and the description of the abstract isomorphisms of the groups of the type $\mathcal G_n(F)$. E.\,Bunina extended their results to unitary linear and Chevalley groups  (see~\cite{Bunina_intro1}, \cite{Bunina_intro3}, \cite{Bunina-local}, \cite{Bunina_intro2}).  Note that in all the results above the first-order theories include only the standard constants from the languages of groups and rings. The model theory of the group $UT_n(R)$, where $n\geqslant 3$, and $R$ is an arbitrary unitary associative ring, was studied in details by O.\,Belegradek~\cite{M3}. He used heavily that the ring $R$ is interpretable
(with parameters) in $UT_n(R)$. A.\,Myasnikov and M.\,Sohrabi studied model theory of groups $\SL_n(\mathcal O)$, $\GL_n(\mathcal O)$, and $T_n(\mathcal O)$ over fields and rings of algebraic integers in~\cite{M55} and \cite{Myasnikov-Sohrabi}.
Their method exploits the mutual interpretability (and also bi-interpretability)
of the group and the ring. In a similar manner N.\,Avni, A.\,Lubotsky, and C.\,Meiri in~\cite{MyasKharl3} 
studied the first order rigidity of non-uniform higher rank arithmetic groups (see also~\cite{AvniMeiri2}). Recently, D.\,Segal and K.\,Tent (see~\cite{Segal-Tent}) showed that for Chevalley groups $G_\pi(\Phi,R)$ of rank
$>1$ over an integral domain~$R$ if $G_\pi(\Phi,R)$ has finite elementary width or is adjoint, then $G_\pi(\Phi,R)$ and $R$
are bi-interpretable.  In~\cite{bunina2022} E.\,Bunina proved that over local rings Chevalley groups $G_\pi(\Phi,R)$ of rank
$>1$ are regularly bi-interpretable with the corresponding rings.

Though related, all the model-theoretic results above do not shed much light on the Diophantine problem in the corresponding groups. Because to relate the Diophantine problems in $\mathcal G_n(R)$ or $G_\pi(\Phi,R)$ and $R$ one needs to have their mutual interpretability by equations, not by arbitrary first-order
formulas.  This is precisely what Myasnikov and Sohrabi did in their paper~\cite{Myasnikov-Sohrabi2} for classical linear groups $\GL_n(R)$, $\SL_n(R)$, $T_n(R)$, $UT_n(R)$ and what we do in this paper for Chevalley groups $G_\pi(\Phi,R)$.

 Recall that a subset (in particular a subgroup) $H$ of a group~$G$ is \emph{Diophantine} in~$G$ if it is definable in~$G$ by a formula of the type 
$$
\Phi(x)= \exists y_1 \dots \exists y_n \left( \bigwedge_{i=1}^k w_i(x,y_1,\dots, y_n)=1\right),
$$
 where $w_i(x,y_1,\dots, y_n)$ is a group word on $x, y_1, \dots,  y_n$. Such formulas are called \emph{Diophantine} (in number theory) or \emph{positive-primitive} (in model theory). Following~\cite{M34}, we say that a structure $\mathcal A$ is \emph{e-interpretable} (or \emph{interpretable by equations}, or \emph{Diophantine interpretable}) in a structure~$\mathcal B$ if $\mathcal A$ is interpretable (see below) in~$\mathcal B$ by Diophantine formulas. The main point of this definition is that if $\mathcal A$ is
e-interpretable in~$\mathcal B$ then the Diophantine problem in~$\mathcal A$ reduces in polynomial
time (\emph{Karp reduces}) to the Diophantine problem in~$\mathcal B$. On the one hand, it
is harder to get e-interpretability than just interpretability, since in the latter you can use arbitrary formulas (not only the Diophantine ones), but on the other hand, to study first-order equivalence of structures one does not usually use the constants in the language, while in the Diophantine problems the constants are required.

A subgroup $G\subseteq \GL_n(R)$ is termed \emph{large} if it contains the subgroup $E_n(R)$
generated in $\GL_n(R)$ by all transvections $t_{ij}(\alpha)$, $i\ne j$, and $\alpha\in R$. In particular, the subgroups $\SL_n(R)$ (when $R$ is commutative) and $E_n(R)$ itself are large. Similarly a subgroup $G\subseteq G_\pi(\Phi,R)$ of a Chevalley group is called large  if it contains the elementary Chevalley group (subgroup) $E_\pi(\Phi,R)$ generated by all elementary unipotents $x_\alpha(t)$, $\alpha\in \Phi$, $t\in R$.
Introduction of  large subgroups of $\GL_n(R)$ ($G_\pi(\Phi,R)$) allows one to unify similar arguments,
otherwise used separately for each of the groups $\GL_n(R)$, $\SL_n(R)$ and $E_n(R)$.
This also emphasize the fact that the methods of the paper~\cite{Myasnikov-Sohrabi2} as well as this paper, unlike the one used in Maltsev's papers~\cite{Maltsev},  is based solely on transvections and nilpotent subgroups (elementary unipotents). Below by  $X_\alpha$, $\alpha\in \Phi$,  we denote the one-parametric subgroup $\{ x_\alpha (t)\mid t\in R\}$.
In Section~5 we   study Diophantine subgroups of large subgroups $G_\pi(R)$.
In particular, we prove the following key technical result (compare with~\cite{Myasnikov-Sohrabi2}):

\begin{result}[Propositions \ref{theorM4.1} and \ref{theorM4.1-B2}] Let $G$ be a large subgroup of $G_\pi(\Phi,R)$, $rank\, \Phi > 1$. Then for any
$\alpha \in \Phi$ the one-parametric subgroup $X_\alpha$ is Diophantine in~$G$ $($defined with constants from the set $\{ x_\alpha(1)\mid \alpha \in \Phi\})$.
\end{result}

This result is similar to the corresponding one from~\cite{Myasnikov-Sohrabi2}, but it mostly based on the description of centralizers of certain sets in all Chevalley groups $G_\pi(\Phi,R)$ of $rank\, \Phi > 1$ over arbitrary commutative rings. This description is proved in Sections 2--4: in Section~2 it is done for  Chevalley groups over all fields, using Bruhat decomposition and direct calculations; in Section 3 it is done for  Chevalley groups over all local rings, using Gauss decomposition, results of the previous section and also direct calculations; in Section 4 it is generalized for all commutative rings, using localization method and results of the previous section. Finally we prove the following result (that has an independent value):

\begin{result}[Theorem \ref{double_centr_arbitrary}]
For any Chevalley group (or its large subgroup) $G=G_\pi(\Phi,R)$, where $\Phi$ is an irreducible root system of a rank $>1$, $R$ is an arbitrary commutative ring with~$1$, if for some $\alpha\in \Phi$ an element $g\in C_G(\Gamma_\alpha)$, then $g=c x_\alpha(t)$, where $t\in R$, $c\in Z(G)$, except the case $\Phi=\mathbf C_l$, $l\geqslant 2$, and $\alpha$ is short.

In the case $\Phi=\mathbf C_l=\{ \pm e_i\pm e_j\mid 1\leqslant i,j\leqslant l,i\ne j\}\cup \{ \pm 2e_i\mid 1\leqslant i\leqslant l\}$ and $\alpha=e_1+e_2$ if $g\in C_G(\Gamma_\alpha)$, then 
$$
g=c x_{e_1+e_2}(t_1)x_{2e_1}(t_2)x_{2e_2}(t_3),\quad c\in Z(G).
$$
\end{result}

Result 1 helps to prove

\begin{result}[Theorem \ref{theorM6.1}]
 Let $G$ be a large subgroup of a Chevalley group $G_\pi(\Phi,R)$,  where $\Phi$ is indecomposable root system of the rank $\ell > 1$, $R$ is an arbitrary commutative rings with~$1$. Then the ring $R$ is
$e$-interpretable in~$G$ (using constants from the set $C_\Phi = \{ x_\alpha(1) \mid \alpha \in \Phi\}$).
\end{result}

This last theorem gives us the result about Karp equivalence of (elementary) Chevalley groups and the correslonding rings:

\begin{result}[Theorems \ref{Theorem6.5} and \ref{Theorem6.8}]
 If $\Phi$ is an indecomposable root system of a rank $> 1$, $R$ is an arbitrary commutative ring with~$1$, then  the Diophantine problem in any Chevalley group
$G_\pi(\Phi,R)$ is Karp equivalent to
the Diophantine problem in~$R$. More precisely:
\begin{itemize}
\item [1)] If $C$ is a countable subset of $G_\pi(\Phi,R)$  then $\mathcal D_C(G_\pi(\Phi,R))$ Karp reduces to $\mathcal D_{R_C} (R)$.
\item [2)] If $T$ is a countable subset of $R$ then there is a countable subset $C_T$ of $G_\pi(\Phi,R)$  such that $\mathcal D_{T} (R)$ Karp reduces to $\mathcal D_{C_T}(G_\pi(\Phi,R))$.
\end{itemize}

 If the elementary Chevalley group $E_\pi(\Phi,R)$ has bounded elementary generation, then the Diophantine problem in 
$E_\pi(\Phi,R)$ is Karp equivalent to
the Diophantine problem in~$R$. 
\end{result}

Section 6 is devoted to applications of the main theorems. 
In~\cite{Myasnikov-Sohrabi2} similar corollaries were proved for classical linear groups, here we repeat them for Chevalley groups.


\begin{result}[Theorem \ref{Theor7.2}]
If $\Phi$ is a indecomposable root system of a rank $>1$, then the Diophantine problem in all Chevalley groups $G_\pi(\Phi,\mathbb Z)$ is Karp equivalent to the Diophantine problem in~$\mathbb Z$, in particular, it is
undecidable.
\end{result}

The following is one of the major conjectures in number theory:

\emph{The Diophantine problem in $\mathbb{Q}$, as well as in any number field~$F$, or any ring of algebraic integers $\mathcal{O}$, is undecidable.}

The following result moves the Diophantine problem in Chevalley groups over number fields or rings of algebraic integers from group theory to number theory. 

\begin{result}[Theorem \ref{Theorem7.2_2}]
Let $\Phi$ be  an indecomposable root system of a rank $>1$ and $R$ either a number field or a ring of algebraic integers. Then the above conjecture holds for $R$ if and only if  the Diophantine problem in the Chevalley group $G_\pi(\Phi,R)$ is undecidable.
\end{result}

The following result from \cite{M34} describes the current state of the Diophantine problem in finitely generated commutative rings:

\emph{ Let $R$ be an infinite finitely generated associative commutative unitary ring.
Then  one of the following holds}:

\begin{enumerate}
    \item \emph{If $R$ has positive characteristic $n> 0$, then the ring of polynomials  $\mathbb{F}_p[t]$ is e-interpretable in $R$ for some transcendental element $t$ and some prime integer $p$; and $\mathcal{D}(R)$ is undecidable.}
    \item  \emph{If  $R$ has zero characteristic and it has infinite rank then the same conclusions as above hold:    the ring of polynomials  $\mathbb{F}_p[t]$ is e-interpretable in $R$ for some $t$ and $p$; and $\mathcal{D}(R)$ is undecidable.} 
    \item  \emph{If  $R$ has zero characteristic and it has finite rank then a ring of algebraic integers $\mathcal{O}$ is e-interpretable in $R$. }
\end{enumerate}

This fact, together with Result~4, implies the following theorem which completely clarifies the situation with the Diophantine problem in Chevalley groups over infinite finitely generated commutative unitary rings:

\begin{result}[Theorem \ref{th:9}]
Let $\Phi$ be an indecomposable root system of a rank $> 1$, $R$ is an arbitrary infinite finitely generated commutative ring with~$1$,  and 
$G_\pi(\Phi,R)$  the corresponding  Chevalley group. Then:

\begin{itemize}
\item [1)]  If $R$ has positive characteristic then the Diophantine problem in $G_\pi(\Phi,R)$ is undecidable.
\item[2)] If  $R$ has zero characteristic and it has infinite rank then the Diophantine problem in $G_\pi(\Phi,R)$ is undecidable. 
\item [3)] If  $R$ has zero characteristic and it has finite rank then the Diophantine problem in some ring of algebraic integers $\mathcal{O}$ is Karp reducible to the Diophantine problem in $G_\pi(\Phi,R)$. Hence if Conjecture \ref{con:majorNT} holds then the Diophantine problem in  $G_\pi(\Phi,R)$ is undecidable.

\end{itemize}

\end{result}

If $R$ is an algebraically closed field,
then

\begin{itemize}
    \item [1)] \emph{If  $A$ is a computable subfield of $R$  then the first-order theory $Th_A(R)$ of $R$ with constants from $A$ in the language  is decidable. In particular, the Diophantine problem $\mathcal{D}_A(R)$  is decidable. } 
    \item [2)] \emph{ If  $A$ is a computable subfield of $R$ then   the algebraic closure $\bar A$ of $A$ in $R$ is  computable.} 

\end{itemize}

Combining the fact above with Theorems \ref{Theorem6.5} and \ref{Theorem6.8}, we obtain

\begin{result}[Theorem \ref{Theorem-alg-closed}]
Let $\Phi$ be an indecomposable root system of a rank $> 1$, $R$ an algebraically closed field, and $G_\pi(\Phi,R)$ the corresponding Chevalley group. If $A$ is a computable subfield of $R$, then the Diophantine problem in $G_\pi(\Phi,R)$ with constants from $G_\pi(\Phi,A)$ is decidable (under a proper enumeration of $G_\pi(\Phi,A)$).
\end{result}

Let  $R = \mathbb{R}$ be  the field of real numbers and $A$  a countable (or finite) subset of $\mathbb{R}$. 
Our treatment of the Diophantine problem in Chevalley groups over $\mathbb{R}$ is based on the following two results on  the Diophantine problem in $\mathbb R$ which are known in the folklore:

\emph{Let $A$ be a finite or countable subset of $\mathbb{R}$.  Then the Diophantine problem in $\mathbb{R}$ with coefficients in $A$ is decidable if and only if the ordered subfield $F(A)$ is computable. Furthermore, in this case the whole first-order theory $Th_A(\mathbb{R})$ is decidable.  }

A real $a \in \mathbb{R}$ is \emph{computable} if  one can effectively approximate it by rationals with any precision. The set of all computable reals $\mathbb{R}^c$ forms a real closed subfield of $\mathbb{R}$, in particular $\mathbb{R}^c$  is first-order equivalent to $\mathbb{R}$. A matrix $A \in \GL_n(\mathbb{R})$ is called \emph{computable} if all entries in $A$ are computable real numbers. 
Chevalley groups $G_\pi(\Phi,\mathbb{R})$ are matrix algebraic groups over $\mathbb{R}$, hence one can view their elements as matrices.

\begin{result}[Theorem \ref{th:main-R}] 
Let $\Phi$ be an indecomposable root system of a rank $> 1$ and $G_\pi(\Phi,\mathbb{R})$ the Chevalley group over the field of real numbers $\mathbb{R}$. If $A$ is a computable ordered subfield of $\mathbb R$ then   the first-order theory  $Th(G_\pi(\Phi,\mathbb{R}))$ with constants from $G_\pi(\Phi,A)$ is decidable. In particular, 
the Diophantine problem in $G_\pi(\Phi,\mathbb{R})$ with constants from $G_\pi(\Phi,A)$ is decidable (under a proper enumeration of $G_\pi(\Phi,A)$).
 \end{result}

 \begin{result}[Theorem \ref{th:main-R-dop}]
Let $\Phi$ be an indecomposable root system of a rank $> 1$ and $G_\pi(\Phi,\mathbb{R}^c)$ the Chevalley group over the field of computable real numbers $\mathbb{R}^c$.  Then the  following holds:
\begin{itemize} 
\item [1)] 
 The Diophantine problem in the computable group $G_\pi(\Phi,\mathbb{R}^c)$  is undecidable. 
 \item [2)] For any finitely generated subgroup $C$ of $G_\pi(\Phi,\mathbb{R}^c)$  the Diophantine problem in $G_\pi(\Phi,\mathbb{R}^c)$ with coefficients in $C$ is decidable. 
 \end{itemize}
 
 \end{result}

 \begin{result}[Theorem \ref{th:incomp-R}]
  Let $\Phi$ be an indecomposable root system of a rank $> 1$ and $G_\pi(\Phi,\mathbb{R})$ the corresponding Chevalley group over the field of computable real numbers $\mathbb{R}^c$. If an element  $g \in E_\pi(\Phi,\mathbb{R})$  is not computable then the  Diophantine problem for equations with coefficients in $\{x_\alpha(1) \mid \alpha \in \Phi\} \cup \{g\}$ is undecidable in any large subgroup of $G_\pi(\Phi,\mathbb{R})$.

\end{result}

Similar to the case of reals  one can define computable $p$-adic numbers for every fixed prime $p$. 

\begin{result}[Theorem \ref{theor-padic-1}]
Let $\Phi$ be an indecomposable root system of a rank $> 1$. Then the  following holds:
 \begin{itemize}
     \item [1)] Let $a_1, \ldots, a_m \in \mathbb{Q}_p^c$ and   $A=\mathbb{Q}(a_1, \ldots,a_m)$  is the subfield of $\mathbb{Q}_p$  generated by $a_1, \ldots, a_m$. Then    the first-order theory  $Th(G_\pi(\Phi,\mathbb{Q}_p))$ with constants from $G_\pi(\Phi,A)$ is decidable. In particular, 
the Diophantine problem in $G_\pi(\Phi,\mathbb{Q}_p)$ with constants from $G_\pi(\Phi,A)$ is decidable (under a proper enumeration of $G_\pi(\Phi,A)$).
     \item [2)] Let $a_1, \ldots, a_m \in \mathbb{Z}_p^c$ and   $A=\mathbb{Z}(a_1, \ldots,a_m)$  is the subring of $\mathbb{Z}_p$  generated by $a_1, \ldots, a_m$. Then    the first-order theory  $Th(G_\pi(\Phi,\mathbb{Q}_p))$ with constants from $G_\pi(\Phi,A)$ is decidable. In particular, 
the Diophantine problem in $G_\pi(\Phi,\mathbb{Q}_p)$ with constants from $G_\pi(\Phi,A)$ is decidable (under a proper enumeration of $G_\pi(\Phi,A)$).
 \end{itemize}
  \end{result}

  \begin{result}[Theorem \ref{theor-padic-2}]
 Let $\Phi$ be an indecomposable root system of a rank $> 1$ and $G_\pi(\Phi,\mathbb{Q}_p)$ $(G_\pi(\Phi,\mathbb{Z}_p), p \neq 2)$ the corresponding Chevalley group over $\mathbb{Q}_p$ $(\mathbb{Z}_p)$. If an element  $g \in E_\pi(\Phi,\mathbb{Q}_p)$ $(g \in E_\pi(\Phi,\mathbb{Z}_p), p \neq 2)$ is not computable then the  Diophantine problem for equations with coefficients in $\{x_\alpha(1) \mid \alpha \in \Phi\} \cup \{g\}$ is undecidable in any large subgroup of $G_\pi(\Phi,\mathbb{Q}_p)$  $(G_\pi(\Phi,\mathbb{Z}_p), p \neq 2)$.   

\end{result}

\section{Chevalley groups}\leavevmode

In this section we establish some notation and recall technical results that are
used throughout the paper.

For a group $G$ and $x,y\in G$ we denote by $x^y$ the conjugate $yxy^{-1}$ of $x$ by~$y$, and
by $[x, y]$ the commutator $xyx^{-1}y^{-1}$. For a subset $A\subseteq G$ by $C_G(A)$ we denote the
centralizer $\{ x\in G\mid [x,a]=1\, \forall a\in A\}$, in particular, $Z(G) = \{ x \in G \mid [x, y] =1\, \forall y\in G\}$
 is the center of~$G$. For subsets $X, Y\subseteq G$ by $[X, Y]$ we denote the subgroup of~$G$ generated by all commutators $[x, y]$, where $x\in X$, $y\in Y$. Then $[G;G]$ is the derived subgroup $G'$ of~$G$ (the \emph{commutant} of~$G$). 

In the rest of the paper by~$R$ we denote an arbitrary associative {\bf commutative} ring with identity~$1$. By $R^*$ we denote the multiplicative group of invertible (unit) elements of~$R$
and by $R^+$ the additive group of~$R$.

\subsection{Root systems and semisimple Lie algebras}\leavevmode

We fix an indecomposable root system~$\Phi$ of the rank $\ell > 1$, with the system of simple
roots~$\Delta$, the set of positive (negative) roots $\Phi^+$
($\Phi^-$), and the Weil group~$W$. Recall that  any two roots of the same length are conjugate under the action of the Weil group. Let $|\Phi^+|=m$. More detailed texts about root systems and their
properties can be found in the books \cite{Hamfris}, \cite{Burbaki}.

Recall also that for $\alpha,\beta\in \Phi$
$$
\langle \alpha,\beta\rangle =2\frac{(\alpha,\beta)}{(\beta,\beta)},
$$
 where $(\alpha,\beta)$ stands for the standard scalar product on the root space.

Suppose now that we have a semisimple complex Lie algebra~$\mathcal
L$ with the Cartan subalgebra~$\mathcal H$ (more details about
semisimple Lie algebras can be found, for instance, in the book~\cite{Hamfris}).

Lie algebra  $\mathcal L$ has a decomposition ${\mathcal
L}={\mathcal H} \oplus \sum\limits_{\alpha\ne 0} {\mathcal
L}_\alpha$,
$$
{\mathcal L}_\alpha:=\{ x\in {\mathcal L}\mid [h,x]=\alpha(h)x\text{
for every } h\in {\mathcal H}\},
$$
and if ${\mathcal L}_\alpha\ne 0$, then $\dim {\mathcal
L}_\alpha=1$, all nonzero $\alpha\in {\mathcal H}$ such that
${\mathcal L}_\alpha\ne 0$, form some root system~$\Phi$. The root
system $\Phi$ and the semisimple Lie algebra
 $\mathcal L$ over~$\mathbb C$
uniquely (up to automorphism) define each other.

On the Lie algebra $\mathcal L$ one can introduce a bilinear
\emph{Killing form} $\varkappa(x,y)=\tr (\ad x\ad y),$  that is
non-degenerated on~$\mathcal H$. Therefore we can identify the
spaces $\mathcal H$ and ${\mathcal H}^*$.

We can choose a basis $\{ h_1, \dots, h_l\}$ in~$\mathcal H$ and for
every $\alpha\in \Phi$ elements $x_\alpha \in {\mathcal L}_\alpha$
so that $\{ h_i; x_\alpha\}$ is a basis in~$\mathcal L$ and for
every two elements of this basis their commutator is an integral
linear combination of the elements of the same basis. This basis is
called a \emph{Chevalley basis}.

\subsection{Elementary Chevalley groups}\leavevmode

Introduce now elementary Chevalley groups (see~\cite{Steinberg}).

Let  $\mathcal L$ be a semisimple Lie algebra (over~$\mathbb C$)
with a root system~$\Phi$, $\pi: {\mathcal L}\to \mathfrak{gl}(V)$
be its finitely dimensional faithful representation  (of
dimension~$n$). If $\mathcal H$ is a Cartan subalgebra of~$\mathcal
L$, then a functional
 $\lambda \in {\mathcal H}^*$ is called a
 \emph{weight} of  a given representation, if there exists a nonzero vector $v\in V$
 (that is called a  \emph{weight vector}) such that
for any $h\in {\mathcal H}$ $\pi(h) v=\lambda (h)v.$

In the space~$V$ in the Chevalley basis all operators
$\pi(x_\alpha)^k/k!$ for $k\in \mathbb N$ are written as integral
(nilpotent) matrices. An integral matrix also can be considered as a
matrix over an arbitrary commutative ring with~$1$. Let $R$ be such
a ring. Consider matrices $n\times n$ over~$R$, matrices
$\pi(x_\alpha)^k/k!$ for
 $\alpha\in \Phi$, $k\in \mathbb N$ are included in $M_n(R)$.

Now consider automorphisms of the free module $R^n$ of the form
$$
\exp (tx_\alpha)=x_\alpha(t)=1+t\pi(x_\alpha)+t^2
\pi(x_\alpha)^2/2+\dots+ t^k \pi(x_\alpha)^k/k!+\dots
$$
Since all matrices $\pi(x_\alpha)$ are nilpotent, we have that this
series is finite. Automorphisms $x_\alpha(t)$ are called
\emph{elementary root elements}. The subgroup in $\Aut(R^n)$,
generated by all $x_\alpha(t)$, $\alpha\in \Phi$, $t\in R$, is
called an \emph{elementary Chevalley group} (notation:
$E_\pi(\Phi,R)$).

In elementary Chevalley group we can introduce the following
important elements and subgroups:

\begin{itemize}
\item $w_\alpha(t)=x_\alpha(t) x_{-\alpha}(-t^{-1})x_\alpha(t)$, $\alpha\in \Phi$,
$t\in R^*$;

\item $h_\alpha (t) = w_\alpha(t) w_\alpha(1)^{-1}$;

\item  $N$ is generated by all
 $w_\alpha (t)$, $\alpha \in \Phi$, $t\in R^*$;

\item  $H$ is generated by all
 $h_\alpha(t)$, $\alpha \in \Phi$, $t\in R^*$;

\item The subgroup $U=U(R)$ of the Chevalley group $G$ ($E$) is generated by elements $x_\alpha(t)$, $\alpha\in \Phi^+$, $t\in R$, the subgroup $V=V(R)$ is generated by elements $x_{-\alpha}(t)$, $\alpha\in \Phi^+$ $t\in R$.
\end{itemize}

The action of  $x_\alpha(t)$ on the Chevalley basis is described in
\cite{v23}, \cite{VavPlotk1}.

It is known that the group $N$ is a normalizer of~$H$ in elementary
Chevalley group, the quotient group $N/H$ is isomorphic to the Weil
group $W(\Phi)$.

All weights of a given representation (by addition) generate a
lattice (free Abelian group, where every  $\mathbb Z$-basis  is also
a $\mathbb C$-basis in~${\mathcal H}^*$), that is called the
\emph{weight lattice} $\Lambda_\pi$.

 Elementary Chevalley groups are defined not even by a representation of the Chevalley groups,
but just by its \emph{weight lattice}. Namely, up to an abstract
isomorphism an elementary Chevalley group is completely defined by a
root system~$\Phi$, a commutative ring~$R$ with~$1$ and a weight
lattice~$\Lambda_\pi$.

Among all lattices we can mark two: the lattice corresponding to the
adjoint representation, it is generated by all roots (the \emph{root
lattice}~$\Lambda_{ad}$) and the lattice generated by all weights of
all reperesentations (the \emph{lattice of weights}~$\Lambda_{sc}$).
For every faithful reperesentation~$\pi$ we have the inclusion
$\Lambda_{ad}\subseteq \Lambda_\pi \subseteq \Lambda_{sc}.$
Respectively, we have the \emph{adjoint} and \emph{universal} (\emph{simply connected})
elementary Chevalley groups. 

Every elementary Chevalley group satisfies the following relations:

(R1) $\forall \alpha\in \Phi$ $\forall t,u\in R$\quad
$x_\alpha(t)x_\alpha(u)= x_\alpha(t+u)$;

(R2) $\forall \alpha,\beta\in \Phi$ $\forall t,u\in R$\quad
 $\alpha+\beta\ne 0\Rightarrow$
$$
[x_\alpha(t),x_\beta(u)]=x_\alpha(t)x_\beta(u)x_\alpha(-t)x_\beta(-u)=
\prod x_{i\alpha+j\beta} (c_{ij}t^iu^j),
$$
where $i,j$ are integers, product is taken by all roots
$i\alpha+j\beta$, taken in some fixed order; $c_{ij}$ are
integer numbers not depending on $t$ and~$u$, but depending on
$\alpha$ and $\beta$ and the order of roots in the product. 

(R3) $\forall \alpha \in \Phi$ $w_\alpha=w_\alpha(1)$;

(R4) $\forall \alpha,\beta \in \Phi$ $\forall t\in R^*$ $w_\alpha
h_\beta(t)w_\alpha^{-1}=h_{w_\alpha (\beta)}(t)$;

(R5) $\forall \alpha,\beta\in \Phi$ $\forall t\in R^*$ $w_\alpha
x_\beta(t)w_\alpha^{-1}=x_{w_\alpha(\beta)} (ct)$, where
$c=c(\alpha,\beta)= \pm 1$;

(R6) $\forall \alpha,\beta\in \Phi$ $\forall t\in R^*$ $\forall u\in
R$ $h_\alpha (t)x_\beta(u)h_\alpha(t)^{-1}=x_\beta(t^{\langle
\beta,\alpha \rangle} u)$.

For a given $\alpha\in \Phi$  by $X_\alpha$ we denote the subgroup   $\{ x_\alpha
(t)\mid t\in R\}$.

\subsection{Chevalley groups over rings}\leavevmode

We briefly recall some basics related to the definition of  Chevalley groups. For more details on Chevalley groups
over rings see  \cite{Steinberg},
\cite{Chevalley}, \cite{v3}, \cite{v23}, \cite{v30}, \cite{Vavilov},
\cite{VavPlotk1}, and references therein.

Let $\Phi$ be a reduced irreducible root system of rank $\geqslant 2$,
and $W=W(\Phi)$ be its Weyl group. Consider a lattice $\Lambda_\pi$ intermediate
between the root lattice $\Lambda_{\ad}$ and the weight
lattice $\Lambda_{sc}$. Let $R$ be a commutative ring with 1,
with the multiplicative group $R^*$.

These data determine the Chevalley group $G=G_{\Lambda_\pi}(\Phi,R)$,
of type $(\Phi,\Lambda_\pi)$ over $R$. It is usually constructed as the
group of $R$-points of the Chevalley--Demazure
group scheme $G_{\Lambda_\pi}(\Phi,\text{$-$})$ of type $(\Phi,\Lambda_\pi)$.
In the case
$\Lambda_\pi=\Lambda_{sc}$ the group $G$ is called \emph{simply connected} and
is denoted by $G_{sc}(\Phi,R)$. In another extreme case
$\Lambda_\pi=\Lambda_{\ad}$ the group $G$ is called \emph{adjoint} and
is denoted by $G_{\ad}(\Phi,R)$. Many results do not depend
on the lattice $\Lambda_\pi$ and hold for all groups of a given
type~$\Phi$. In all such cases, or when $\Lambda_\pi$ is determined by
the context, we omit any reference to $\pi$ in the notation
and denote by $G(\Phi,R)$ {\it any} Chevalley group of type
$\Phi$ over $R$. 
\par

In what follows, we also fix a split maximal torus $T=T(\Phi,R)$ in $G=G(\Phi,R)$.  This choice
uniquely determines the unipotent root subgroups, $X_{\alpha}$,
$\alpha\in\Phi$, in $G$, elementary with respect to $T$. As usual,
we fix maps $x_{\alpha}\colon R\mapsto X_{\alpha}$, so that
$X_{\alpha}=\{x_{\alpha}(t)\mid t\in R\}$, and require that these parametrizations are interrelated by the Chevalley commutator formula with integer coefficients, see R1--R2. The above unipotent elements
$x_{\alpha}(t)$, where $\alpha\in\Phi$, $t\in R$,
elementary with respect to $T(\Phi,R)$, are also called
elementary unipotent root elements or, for short, simply
root unipotents.
\par
Further,
$$ E(\Phi,R)=\big\langle x_\alpha(t),\ \alpha\in\Phi,\ t\in R\big\rangle $$
\noindent
denotes the  elementary subgroup of $G(\Phi,R)$,
spanned by all elementary root unipotents, or, what is the
same, by all root subgroups $X_{\alpha}$,
$\alpha\in\Phi$.
\par
This is precisely the elementary Chevalley group defined in the previous section. 
One can look at Chevalley groups also from the positions of algebraic groups. This point of view is of special importance for many application.

All these groups are defined in $\SL_n(R)$ as  common set of zeros of
polynomials of matrix entries $a_{ij}$ with integer coefficients
 (for example,
in the case of the root system $\mathbf C_\ell$ and the universal
representation we have $n=2l$ and the polynomials from the condition
$(a_{ij})Q(a_{ji})-Q=0$). It is clear now that multiplication and
taking inverse element are also defined by polynomials with integer
coefficients. Therefore, these polynomials can be considered as
polynomials over arbitrary commutative ring with a unit. Let some
elementary Chevalley group $E$ over~$\mathbb C$ be defined in
$\SL_n(\mathbb C)$ by polynomials $p_1(a_{ij}),\dots, p_m(a_{ij})$.
For a commutative ring~$R$ with a unit let us consider the group
$$
G(R)=\{ (a_{ij})\in \SL_n(R)\mid \widetilde p_1(a_{ij})=0,\dots
,\widetilde p_m(a_{ij})=0\},
$$
where  $\widetilde p_1(\dots),\dots \widetilde p_m(\dots)$ are
polynomials having the same coefficients as
$p_1(\dots),\dots,p_m(\dots)$, but considered over~$R$.

Semisimple linear algebraic groups over algebraically closed fields $K$  are precisely
 Chevalley groups $G(K)= E(\Phi,K)$ (see.~\cite{Steinberg}, \S\,5).

The standard maximal torus of the Chevalley group $G_\pi(\Phi,R)$ is  denoted usually by $T_\pi(\Phi,R)$ and
is isomorphic to $\Hom(\Lambda_\pi, R^*)$.

Let us denote by $h(\chi)$ the elements of the torus $T_\pi
(\Phi,R)$, corresponding to the homomorphism $\chi\in Hom
(\Lambda(\pi),R^*)$.

In particular, $h_\alpha(u)=h(\chi_{\alpha,u})$ ($u\in R^*$, $\alpha
\in \Phi$), where
$$
\chi_{\alpha,u}: \lambda\mapsto u^{\langle
\lambda,\alpha\rangle}\quad (\lambda\in \Lambda_\pi).
$$

\subsection{Connection between Chevalley groups and their elementary subgroups}\leavevmode

Connection between Chevalley groups and corresponding elementary
subgroups is an important problem in the theory of Chevalley
groups over rings. For elementary Chevalley groups there exists a
convenient system of generators $x_\alpha (\xi)$, $\alpha\in \Phi$,
$\xi\in R$, and all relations between these generators are well-known.
For general Chevalley groups it is not always true.

If $R$ is an algebraically closed field, then
$$
G_\pi (\Phi,R)=E_\pi (\Phi,R)
$$
for any representation~$\pi$. This equality is not true even for the
case of fields, which are not algebraically closed.

However if $G$ is a simply connected group and the ring $R$ is \emph{semilocal}
(i.e., contains only finite number of maximal ideals), then we have
the property
$$
G_{sc}(\Phi,R)=E_{sc}(\Phi,R).
$$
\cite{M}, \cite{Abe1}, \cite{St3}, \cite{AS}.

Let us show the difference between Chevalley groups and their
elementary subgroups in the case when a ring $R$ is semilocal and
a corresponding Chevalley group is not simply connected.  In this case $G_\pi
(\Phi,R)=E_\pi(\Phi,R)T_\pi(\Phi,R)$] (see~\cite{Abe1}, \cite{AS},
\cite{M}), and the elements $h(\chi)$ are connected with elementary
generators by the formula
\begin{equation}\label{e4}
h(\chi)x_\beta (\xi)h(\chi)^{-1}=x_\beta (\chi(\beta)\xi).
\end{equation}

\begin{observation}\label{Remark_torus}
Since $\chi\in \Hom (\Lambda(\pi),R^*)$, if we know the values of~$\chi$ on some set of roots which generate all roots (for example, on some basis of~$\Phi$), then we know $\chi(\beta)$ for all $\beta\in \Phi$ and respectively all $x_\beta(\xi)^{h(\chi)}$ for all $\beta\in \Phi$ and $\xi\in R^*$.

Therefore in particular if for all roots $\beta$ from some generating set of~$\Phi$ we have $[x_\beta(1),h(\chi)]=1$, then $h(\chi)\in Z(E_\pi(\Phi, R)$ and hence $h(\chi)\in Z(G_\pi(\Phi,R)$.

We will use this observation in the next section many times.
\end{observation}

If $\Phi$ is an irreducible root system of a rank $\ell\geqslant
2$, then $E(\Phi,R)$ is always normal and even {\bf characteristic} in $G(\Phi,R)$ (see~\cite{Tadei}, \cite{Hasrat-Vavilov}). In the case of
semilocal rings it is easy to show that
$$
[G(\Phi,R),G(\Phi,R)]=E(\Phi,R).
$$
except the cases $\Phi=\mathbf B_2, \mathbf G_2$, $R=\mathbb F_2$.

However in the case $\ell=1$ the subgroup of elementary matrices
$E_2(R)=E_{sc}(\mathbf A_1,R)$ is not necessarily normal in the special linear
group $\SL_2(R)=G_{sc}(\mathbf A_1,R)$ (see~\cite{Cn}, \cite{Sw},
\cite{Su1}).

In the general case the difference between $G_\pi(\Phi,R)$ and $E_\pi (\Phi,R)$ is measured by $K_1$-functor.

\section{The Diophantine problem}\leavevmode

\subsection{Equations, constants and computable structures}\leavevmode

Recall, that the Diophantine problem in an algebraic structure~$\mathcal A$ (denoted
$\mathcal D(\mathcal A)$) is the task to determine whether or not a given finite system of equations with constants in~$\mathcal A$ has a solution in~$\mathcal A$. $\mathcal D(\mathcal A)$ is decidable if there is an
algorithm that given a finite system $S$ of equations with constants in~$\mathcal A$ decides whether or not $S$ has a solution in~$\mathcal A$. Here, the structure $\mathcal A$ is assumed to be countable, moreover, supposedly it comes equipped with a fixed enumeration $\mathcal A = \{ a_1, a_2, \dots\}$, which is given by a surjective function $\nu: \mathbb N\to \mathcal A$ (the function is not necessary injective). One can use the function for enumeration of all finite systems of equations with coefficients in~$\mathcal A$ in countably many
variables $x_1, x_2,\dots$, and then provide them as inputs to a decision algorithm in the Diophantine problem 
$\mathcal D(\mathcal A)$. The first question to address here is how much decidability of $\mathcal D(\mathcal A)$ 
depends on the choice of the enumeration $\nu: \mathbb N\to \mathcal A$. It turns out, that decidability of $\mathcal D(\mathcal A)$ does depend on the enumeration~$\nu$, so for some~$\nu$, $\mathcal D(\mathcal A)$ can be decidable, and for 
others can be not. For example, every non-trivial finite or countable group has an infinite countable presentation with undecidable word problem, so the Diophantine problems in the group with respect to the enumerations related to such 
infinite presentations are undecidable. However, researchers are usually interested only in ``natural'' enumerations~$\nu$, which come from finite descriptions of the elements of~$\mathcal A$ that reflect the nature of 
the structure~$\mathcal A$. For instance, if $\mathcal A$ is a finitely generated group then one may describe elements of~$\mathcal A$ by finite words in a fixed finite set of generators, and use known effective enumerations 
of words, while if $\mathcal A$ is, say, a group $\GL_n(R)$ over a ring~$R$, then elements of $\GL_n(R)$ can be described by $n^2$-tuples of elements from~$R$, so one can use enumerations of~$R$ to enumerate elements 
of $\GL_n(R)$. Here, and in all other places, by an effective enumeration of words (or polynomials, or any other formulas of finite signature) we understand such an enumeration $\mu: n\mapsto w_n$ of words in a given finite or 
countable alphabet that for any number $n\in \mathbb N$ one can compute the word $w_n$ and for any word~$w$ in the given alphabet one can compute a number $n$ such that $w=w_n$. If $\mathcal A$ is a finitely 
generated associative unitary ring $R$ then elements of~$\mathcal A$ can be presented as non-commutative polynomials with integer coefficients in finitely many variables (which can be also viewed as elements of a free 
associative unitary ring of finite rank) and then effectively enumerated. Similarly, for commutative 
rings~$R$ the usual commutative polynomials can be used. There are two ways to make the formulation of the 
Diophantine a bit more precise, either explicitly fix the enumeration $\nu$ of~$\mathcal A$ in the Diophantine problem (denote it by $\mathcal D_\nu(\mathcal A)$), or to term that $\mathcal D(\mathcal A)$ is decidable if \emph{there exists} an enumeration $\nu$ of~$\mathcal A$ such that $\mathcal D(\mathcal A)$ is decidable. To study which enumerations are ``reasonable'' in the discourse of Diophantine problems we need to digress to the theory of computable algebra, or computable model theory, that stem from pioneering works of Rabin~\cite{M62} and Maltsev~\cite{M50} (for details see a book~\cite{M27} and a more recent survey~\cite{M29}).

Recall that a structure $\mathcal A$ of finite signature is \emph{computable} with respect to an enumeration
$\nu: \mathbb N\to \mathcal A$ if all the basic operations and predicates (including the equality) on~$\mathcal A$ 
are computable with respect to the enumeration~$\nu$. In particular, a group $G$ is \emph{computable} with 
respect to~$\nu$, if there are two computable functions $f(x,y)$ and $h(x,y)$ such that for any $i, j\in \mathbb N$ 
the following holds: $\nu(i)\cdot \nu(h)=\nu(f(i,j))$ and $\nu(i)=\nu(j)\Longleftrightarrow h(i,j)=1$. Similarly, a
countable ring $R$ is computable with respect to enumeration $\nu:\mathbb N\to R$ if in addition to the conditions above there is a computable function $g(x, y)$ such that $\nu(i)+\nu(j)=\nu(g(i, j))$.

The following observation shows the connection between decidability of Diophantine problems and computable structures.

\begin{lemma}\label{lemmaM3.1} Let $\mathcal A$ be a countable structure given with an enumeration $\nu: \mathbb N\to \mathcal A$. If the Diophantine problem $\mathcal D_\nu(\mathcal A)$ is decidable then the structure $\mathcal A$ is computable with respect to~$\nu$.
\end{lemma}

Lemma~\ref{lemmaM3.1} shows that  the only interesting enumerations of $\mathcal A$ with respect
to the Diophantine problem are those that make $\mathcal A$ computable, they are called
\emph{constructivizations} of~$\mathcal A$. The question whether a given countable structure $\mathcal A$
has a constructivization is a fundamental one in computable model theory, so there are a lot of results in this direction (see \cite{M27}, \cite{M29}, \cite{M28}) that can be used here.

Let $\mu$ and $\nu$ be two enumerations of~$\mathcal A$. By definition $\mu$ \emph{reduces} to~$\nu$ (symbolically $\preceq$) if there is a computable function $f(x)$ such that $\mu=\nu\circ f$. Furthermore,  $\mu$ and $\nu$ are
termed \emph{equivalent} (symbolically~$\sim$) if $\mu \preceq \nu$ and $\nu \preceq \mu$.

\begin{lemma}[\cite{M27}]\label{lemmaM3.2}
 Let $\mathcal A$ be a finitely generated structure that have at least one constructivization. Then all constructivizations of~$\mathcal A$ are equivalent to each other.
\end{lemma}

It follows that a finitely generated structure $\mathcal A$ has a constructivization if and only if the word problem in~$\mathcal A$ with respect to some (any) finite generating set is decidable. In this case, any other constructivization is equivalent to the one that comes as described above from any fixed finite set of generators. This is why for finitely generated structures the enumerations usually are not mentioned
explicitly.

If $\mathcal A$ is uncountable then, as we mentioned in Introduction, one has to consider only equations with constants from a fixed countable (or finite) subset $C$ of~$\mathcal A$ which comes equipped with an enumeration $\nu: \mathbb N\to C$. This form of the Diophantine problem is denoted by $\mathcal D_C(\mathcal A)$. It will be convenient to consider instead of the set~$C$ the substructure $\langle C\rangle$ generated by~$C$ in~$\mathcal A$. In this case one needs to consider enumerations of $\langle C\rangle$ that are ``compatible'' with the given enumeration of~$C$. To this end we introduce the following notion from computable model theory (see~\cite{M27}). Let $S$ be a set with an enumeration~$\nu:\mathbb{N} \to S$ and $\varphi: S\to S^*$ an embedding of sets. We say that an enumeration $\nu^*: \mathbb N\to S^*$
\emph{extends} the enumeration~$\nu$ if there exists a computable function $f:\mathbb N\to \mathbb N$
such that $\varphi\circ \nu = \nu^\ast\circ f$. It is easy to construct an enumeration of $\langle C\rangle$ that
extends a given enumeration of the generating set~$C$ (see~\cite{M27}, Ch.\,6, Section\,1, Theorem\,1). 
In the case of the subset $C$ of~$\mathcal A$ we will always, if not said otherwise, consider enumerations $\nu^\ast: \mathbb N \to \langle C\rangle$ that extend a given enumeration $\nu: \mathbb N \to C$. Furthermore, we will always assume that for a given $n\in \mathbb N$ one can compute the term $t$ of the language of the structure~$\mathcal A$ with constants from~$C$ which represents the element $\nu^*(n)$ in the structure $\langle C\rangle$. And conversely, for every  term~$t$ in the language of~$\mathcal A$ with constants from $C$ one can compute a number $n\in \mathbb N$ such that $\nu^*(n)=t$. We call
such enumerations $\nu^*$  \emph{effective}. To construct an effective enumeration of $\langle C\rangle$ in
the case when $\mathcal A$ is a group one needs only effectively enumerate all words in
the alphabet $C^{\pm 1}$, while in the case when $\mathcal A$ is a commutative unitary ring one
needs to enumerate all polynomials from $\mathbb Z[C]$.

The following lemma from~\cite{Myasnikov-Sohrabi2} is useful.

\begin{lemma}\label{lemmaM3.3}
 Let $\mathcal A$  be a structure, $C$ a finite or countable subset of~$\mathcal A$ equipped with an enumeration~$\nu$, and $\langle C\rangle$ the substructure generated by~$C$ in~$\mathcal A$ with an
effective enumeration that extends~$\nu$. Then the following hold:

\emph{1)} The Diophantine problems $\mathcal D_C(\mathcal A)$ and $\mathcal D_{\langle C\rangle}(\mathcal A)$ are equivalent (reduce to each other).

\emph{2)} If $\mathcal D_C(\mathcal A)$ is decidable then $\langle C\rangle$ is computable with respect to any enumeration of $\langle C\rangle$ that extends the enumeration $\nu$ of the generating set~$C$.
\end{lemma}

From now on we will always assume, without loss of generality, that coefficients in the Diophantine problem is taken from a countable substructure $C$ rather than from the set~$C$.

\subsection{Diophantine sets and e-interpretability}\leavevmode

To prove that $\mathcal D(\mathcal A)$ reduces to $\mathcal D(\mathcal M)$ for some structures $\mathcal A$ and $\mathcal M$ it suffices to show that $\mathcal A$ is \emph{interpretable by equations} (or \emph{e-interpretable}) in~$\mathcal M$.

The notion of e-interpretability was introduced in \cite{M32}, \cite{M31}, \cite{M33}. Here we remind
this notion and state some basic facts we use in the sequel.

In what follows we often use non-cursive boldface letters to denote tuples of
elements: e.g. $\mathbf a = (a_1,\dots, a_n)$. Furthermore, we always assume that equations
may contain constants from the algebraic structure in which they are considered.

\begin{definition}
 A subset $D \subset M^m$ is called \emph{Diophantine}, or \emph{definable by
systems of equations} in~$\mathcal M$, or \emph{e-definable} in~$\mathcal M$, if there exists a finite system
of equations, say $\Sigma_D(x_1,\dots, x_m, y_1,\dots,y_k)$, in the language of~$\mathcal M$ such that for
any tuple $\mathbf a\in M^m$, one has that $\mathbf a \in D$ if and only if the system $\Sigma_D(\mathbf a,\mathbf y)$ on variables~$\mathbf y$ has a solution in~$\mathcal M$. In this case $\Sigma_D$ is said to \emph{e-define $D$ in~$\mathcal M$}.
\end{definition}

\begin{observation}
 Observe that, in the notation above, if $D\subset M^m$ is e-definable
then it is definable in~$\mathcal M$ by the formula $\exists \mathbf y \Sigma_D (\mathbf x, \mathbf y)$.
Such formulas are called \emph{positive primitive}, or \emph{pp-formulas}. Hence, e-definable subsets are sometimes called \emph{pp-definable}. On the other hand, in number theory such sets are usually
referred to as Diophantine ones. And yet, in algebraic geometry they can be described as projections of algebraic sets.
\end{observation}

\begin{definition}
 An algebraic structure $\mathcal A = (A; f,\dots , r,\dots, c,\dots )$ is called \emph{e-interpretable} in another algebraic structure~$\mathcal M$ if there exists $n\in \mathbb N$, a subset
$D\subseteq M^n$ and an onto map (called the \emph{interpreting map}) $\varphi: D\to \mathcal A$, such that:

1. $D$ is e-definable in~$\mathcal M$.

2. For every function $f = f(x_1,\dots, x_n)$ in the language of~$\mathcal A$, the preimage by~$\varphi$ of the graph of~$f$, i.\,e. the set 
$$
\{(x_1,\dots, x_k, x_{k+1})\mid \varphi(x_{k+1})=f(x_1,\dots, x_k)\},
$$
 is e-definable in~$\mathcal M$.

3. For every relation $r$ in the language of~$\mathcal A$, and also for the equality relation $=$ in~$\mathcal A$, the preimage by~$\varphi$ of the graph of~$r$ is e-definable in~$\mathcal M$.
\end{definition}

Let $\mathcal A$ be e-interpretable in~$\mathcal M$ as in definition above. This interpretation is completely determined by the map $\varphi$ and a tuple $\Gamma$ of the Diophantine formulas that are defining the set $D$ from 1), the functions $f$ from 2), and the relations $r$ from 3). By $P_\Gamma \subseteq \mathcal M$ we denote the finite set of constants (parameters) that occur in formulas from~$\Gamma$. E-interpretability is a variation of the classical notion of the first-order interpretability, where instead of arbitrary first-order formulas finite systems of equations are used as the interpreting formulas.

The following is a fundamental property of e-interpretability. Intuitively it states that if $\mathcal A$ is e-interpretable in~$\mathcal M$ by formulas~$\Gamma$ and an interpreting map $\varphi: D\to \mathcal A$, then any system of equations in~$\mathcal A$ can be effectively ``encoded'' by an equivalent system of equations in~$\mathcal M$. To explain we need the following notation. Let $C$ be a finite or countable subset of~$\mathcal A$ equipped with an enumeration $\nu: \mathbb N \to C$. For every $c_i=\nu(i)\in C$ fix an arbitrary tuple $d_i\in \varphi^{-1}(c_i)$. Denote by $D_R$ the set of all elements in~$\mathcal M$ that occur as components in tuples~$d_i$ from~$R$. Denote by~$C_\Gamma$ the set $D_R\cup P_\Gamma$. We say that enumeration
$\nu^*: \mathbb N\to C$ is \emph{compatible} with the enumeration~$\nu$ (with respect to the set of representatives~$R$) if there is an algorithm that for every $i\in \mathbb N$ computes the $\nu^*$-numbers of the
components of the tuple~$d_i$. For example, one can enumerate first all elements in~$P_\Gamma$ and then for $i = 1,2,\dots$ enumerate in the natural order all the components of $d_1,d_2,\dots$.

\begin{lemma}[\cite{M32}]\label{lemmaM3.7}
 Let $\mathcal A$ be e-interpretable in~$\mathcal M$ by a set of formulas~$\Gamma$ with an
interpreting map $\varphi:D\to \mathcal A$. Let $C$ be a finite or countable subset of~$\mathcal A$ equipped with an enumeration~$\nu$. Then there is a polynomial time algorithm that for every finite system of equations $S(\mathbf x)$ in~$\mathcal A$ with coefficients in~$C$ constructs a finite system of equations $S^*(\mathbf y,\mathbf z)$ in~$\mathcal M$ with coefficients in~$C_\Gamma$ \emph{(}given via a compatible enumeration $\nu^*:\mathbb N\to C_\Gamma)$, such that if $(\mathbf b,\mathbf c)$ is a solution to $S^*(\mathbf y;\mathbf z)$ in~$\mathcal M$, then $\mathbf b\in D$ and $\varphi(\mathbf b)$ is a solution to $S(\mathbf x)$ in~$\mathcal A$. Moreover, any solution $\mathbf a$ to $S(\mathbf x)$ in~$\mathcal A$ arises in this way, i.\,e. $\mathbf a=\varphi(\mathbf b)$ for some solution $(\mathbf b;\mathbf c)$ to $S^*(\mathbf y,\mathbf z)$ in~$\mathcal M$.
\end{lemma}

Now we show two key consequences of Lemma~\ref{lemmaM3.7}.

\begin{corollary}\label{corolM3.8}
 Let $\mathcal A$ be e-interpretable in~$\mathcal M$ by a set of formulas~$\Gamma$ with an interpreting map $\varphi:D\to \mathcal A$. Let $C$ be a finite or countable subset of~$\mathcal A$ equipped with an enumeration~$\nu$. Then the Diophantine problem in~$\mathcal A$ with coefficients in~$C$ is reducible in polynomial time (Karp reducible) to the Diophantine problem in~$\mathcal M$ with coefficients in~$C_\Gamma$ with respect to any compatible with~$\nu$ enumeration~$\nu^*$. Consequently, if $\mathcal D_C(\mathcal A)$ is undecidable, then $\mathcal D_{C_\Gamma} (\mathcal M)$ (relative to~$\nu^*$) is undecidable as well.
\end{corollary}

\begin{corollary}\label{corolM3.9}
 e-interpetability is a transitive relation, i.\,e., if $\mathcal A_1$ is e-intepretable in~$\mathcal A_2$, and $\mathcal A_2$ is e-interpretable in~$\mathcal A_3$, then $\mathcal A_1$ is e-interpretable in~$\mathcal A_3$.
\end{corollary}


\section{Double centralizers of uniponent elements in Chevalley groups}\leavevmode

For the purposes of our paper, we want to prove that in the Chevalley group or its large subgroup $G$ over arbitrary commutative ring $R$ with unity, the subgroup $C \cdot  X_\alpha$, where $C=Z(G)$, is a centralizer of some  finite set of elementary unipotents of the Chevalley group. This result holds for all root systems, except for the short roots of  $\mathbf C_l$, where the answer will be slightly different.

Quite similar problems were considered in a number of papers (see \cite{Segal-Tent}, \cite{bunina2022}, \cite{Vavilov_Plotkin_new}). We cannot take these results for granted, since for our aims we need a version of double centralizer theorem for arbitrary commutative rings. With this end we use a localization method, which requires conjugations with the elements of the form $x_\alpha(1)$ only.

\subsection{Localization of rings and modules; injection of a ring into the product of its localizations.}\leavevmode

\begin{definition}  Let $R$ be a commutative ring. A subset $S\subset R$ is called \emph{multiplicatively closed} in~$R$, if $1\in S$ and $S$ is closed under multiplication.
\end{definition}

Introduce  an equivalence relation $\sim$ on the set of pairs $R\times S$ as follows:
$$
\frac{a}{s}\sim \frac{b}{t} \Longleftrightarrow \exists u\in S:\ (at-bs)u=0.
$$
  By $\frac{a}{s}$ we denote the whole equivalence class of the pair $(a,s)$, by $S^{-1}R$ we denote the set of all equivalence classes. On the set $S^{-1}R$ we can introduce the ring structure by
$$
\frac{a}{s}+\frac{b}{t}=\frac{at+bs}{st},\quad \frac{a}{s}\cdot \frac{b}{t}=\frac{ab}{st}.
$$

\begin{definition}
The ring $S^{-1}R$ is called the \emph{ring of fractions of~$R$ with respect to~$R$}.
\end{definition}

 Let $\mathfrak p$ be a prime ideal of~$R$. Then the set $S=R\setminus {\mathfrak p}$ is multiplicatively closed (it is equivalent to the definition of the prime ideal). We will denote the ring of fractions  $S^{-1}R$ in this case by $R_{\mathfrak p}$. The elements $\frac{a}{s}$, $a\in \mathfrak p$, form an ideal $\mathfrak M$ in~$R_{\mathfrak p}$. If $\frac{b}{t}\notin \mathfrak M$, then $b\in S$, therefore $\frac{b}{t}$ is invertible in~$R_{\mathfrak p}$. Consequently the ideal $\mathfrak M$ consists of all non-invertible elements of the ring~$R_{\mathfrak p}$, i.\,e., $\mathfrak M$ is the greatest ideal of this ring, so $R_{\mathfrak p}$ is a local ring.

The process of passing from~$R$ to~$R_{\mathfrak p}$ is called  \emph{localization at~${\mathfrak p}$.}

\begin{proposition}\label{inlocal}
Every commutative ring  $R$ with $1$ can be naturally embedded in the cartesian product of all its localizations  by maximal ideals
 $$
S=\prod\limits_{{\mathfrak m}\text{ is a maximal ideal of }R} R_{\mathfrak m}
$$
by diagonal mapping, which assigns every $a\in R$ to the element
$$
\prod\limits_{\mathfrak m} \left( \frac{a}{1}\right)_{\mathfrak m}
$$
of~$S$.
\end{proposition}

\subsection{Double centralizers of unipotent elements in Chevalley groups over fields}\leavevmode

In the next sections we use relations between elements from Chevalley groups from~\cite{Steinberg} without special notice.

\begin{definition}\label{def_gamma_alpha}
For any Chevalley group $G_\pi(\Phi,R)$ and for any $\alpha\in \Phi$ let 
$$
\Gamma_\alpha=\{ x_\beta(1)\mid \beta\in \Phi\text{ and }[x_\beta(1),x_\alpha(1)]=e\}.
$$
\end{definition}

Recall also that by $C_G(M)$ we denote the centralizer of the set~$M$ in the group~$G$.

The goal of this section is to prove the following theorem, which can be viewed as a variant of the double centralizers theorem:

\begin{theorem}[compare with~\cite{Segal-Tent}]
For any Chevalley group (or its large subgroup) $G=G_\pi(\Phi,\mathbb F)$, where $\mathbb F$ is an arbitrary field, $\Phi$ is an irreducible root system of a rank $>1$, if some element $g\in C_G(\Gamma_\alpha)$, then $g=c x_\alpha(t)$, where $t\in \mathbb F$, $c\in Z(G)$, except the case $\Phi=\mathbf C_l$, $l\geqslant 2$, and $\alpha$ is short.

In the case $\Phi=\mathbf C_l=\{ \pm e_i\pm e_j\mid 1\leqslant i,j\leqslant l,i\ne j\}\cup \{ \pm 2e_i\mid 1\leqslant i\leqslant l\}$ and $\alpha=e_1+e_2$ if $g\in C_G(\Gamma_\alpha)$, then 
$$
g=c x_{e_1+e_2}(t_1)x_{2e_1}(t_2)x_{2e_2}(t_3),\quad c\in Z(G).
$$
\end{theorem}

\begin{observation}
Note that if $\Phi=\mathbf C_l$ and $\alpha$ is an arbitrary short root, then it is always conjugate to $e_1+e_2$ by some element of~$W$. Therefore we do not lose any generality, considering the root $\alpha=e_1+e_2$.
\end{observation}

\smallskip

We will use in this section the Bruhat decomposition
$$
G=BWB=\bigcup_{w\in W}BwB
$$
in the following form

\begin{proposition}[see~\cite{Steinberg}]
Any element  $g\in G$ can be uniquely represented in the form
$$
g=t x_{\alpha_1}(a_1)\dots x_{\alpha_m}(a_m)\mathbf w x_{\alpha_1}(b_1)\dots x_{\alpha_m}(b_m),\text{ where }t\in T, \mathbf w\in W, 
$$
$\alpha_1,\dots,\alpha_m$ are all positive roots of~$\Phi$ with an arbitrary fixed order, $b_i=0$ if $\mathbf w(\alpha_i)\in \Phi^+$.
\end{proposition}
 
Since all roots of the same length are conjugate by the action of~$W$, we can enumerate simple roots of our root system~$\Phi$ so that  $\alpha=\alpha_1$  (the first simple root) and all simple roots $\alpha_3,\dots, \alpha_l$ are orthogonal to~$\alpha_1$.

Also we suppose that $\alpha_1,\dots,\alpha_m$ have non-decreasing   heights, i.\,e. we start from simple roots $\alpha_1,\alpha_2,\dots, \alpha_l$, then their sums etc.

In this situation the roots $\alpha_3,\dots,\alpha_l$ are orthogonal to~$\alpha_1$, therefore 
$$
\forall i=3,\dots, l\ [x_{\pm \alpha_i}(1),x_{\alpha_1}(1)]=1\Longrightarrow x_{\pm \alpha_i}(1)\in \Gamma_\alpha.
$$
Also we know that for the maximal root $\gamma$ of~$\Phi$ 
$$
\gamma+\alpha_1\notin \Phi \Longrightarrow x_\gamma(1)\in \Gamma_\alpha.
$$
Besides that,
$$
x_{\alpha_1}(1), x_{-\alpha_2}(1)\in \Gamma_\alpha.
$$

Suppose that we have some $g$ in Bruhat decomposition
$$
g=t x_{\alpha_1}(a_1)\dots x_{\alpha_m}(a_m)\mathbf w x_{\alpha_1}(b_1)\dots x_{\alpha_m}(b_m)\in C_G(\Gamma_\alpha).
$$

Our first and the most important goal is to prove that $\mathbf w=\mathbf e_W$. 

\begin{lemma}\label{lemma_major_root}
Given the maximal root~$\gamma\in \Phi$, $\mathbf w(\gamma)=\gamma.$
\end{lemma}

\begin{proof}
First of all, it is useful  to have formulas of conjugation for different elements by the element $x_\beta(1)$.

For $t\in T$ we have
$$
t^{x_\beta(1)}=tt^{-1}x_\beta(1)tx_\beta(-1)=t x_\beta(c^\beta_t)x_\beta(-1)=tx_\beta(c^\beta_t-1),\quad c^\beta_t\in \mathbb F^*.
$$
For $\mathbf w\in W$ we have
$$
\mathbf w^{x_\beta(1)}=x_\beta(1) \mathbf w x_\beta(-1)\mathbf w^{-1}\mathbf w=x_\beta(1)x_{\mathbf w(\beta)}(-1) \mathbf w.
$$
In our case of the maximal root $\gamma$ we know that $[x_{\alpha_i}(1),x_\gamma(1)]=1$ for all $i=1,\dots,m$, therefore
\begin{multline*}
g^{x_\gamma(1)}=tx_\gamma(c^\gamma_t-1)x_{\alpha_1}(a_1)\dots x_{\alpha_m}(a_m)x_\gamma(1)x_{\mathbf w(\gamma)}(-1)\mathbf w x_{\alpha_1}(b_1)\dots x_{\alpha_m}(b_m)=\\
=tx_{\alpha_1}(a_1)\dots x_{\alpha_m}(a_m) x_\gamma(c_t^\gamma)x_{\mathbf w(\gamma)}(-1)\mathbf w x_{\alpha_1}(b_1)\dots x_{\alpha_m}(b_m)=\\
=g=tx_{\alpha_1}(a_1)\dots x_{\alpha_m}(a_m) \mathbf w x_{\alpha_1}(b_1)\dots x_{\alpha_m}(b_m).
\end{multline*}
Using uniqueness of Bruhat decomposition and comparing the parts of the equality we see that:

(1) $\mathbf w(\gamma)=\gamma$, what was required; 

(2) $c^\gamma_t=1$, i.\,e. $t$ commutes with $x_\gamma(1)$.
\end{proof}

Recall that the set of simple roots $\Delta$ consists of $\alpha_1,\dots, \alpha_l$. We suppose that the roots $\alpha_3,\dots, \alpha_l$ are orthogonal to~$\alpha_1$. 

\begin{lemma}\label{lemma_roots_positive}
For any root $\beta\in \{ \alpha_1,\alpha_3,\dots, \alpha_l\}$ we have $\mathbf w(\beta)\in \Phi^+$.
\end{lemma}

\begin{proof}
Suppose that $\beta\in \{ \alpha_1,\alpha_3,\dots, \alpha_l\}$, $\beta=\alpha_i$, it is clear that in this case $x_\beta(1)\in \Gamma_\alpha$.

Suppose that $\mathbf w(\beta)=\delta\in \Phi^-$.
Then, using (R3) relation,
\begin{multline*}
tx_{\alpha_1}(a_1)\dots x_{\alpha_m}(a_m) \mathbf w x_{\alpha_1}(b_1)\dots x_{\alpha_m}(b_m)=g=\\
=g^{x_\beta(1)}
=tx_\beta(c^\beta_t-1)x_{\alpha_1}( a_1)\dots x_{\alpha_l}(a_l) x_{\alpha_{l+1}}(\widetilde a_{l+1})\dots x_{\alpha_m}(\widetilde a_m)x_\beta(1)\mathbf w \cdot \\
\cdot x_\beta(-1) x_{\alpha_1}(b_1)\dots  x_{\alpha_l} (b_l)x_{\alpha_{l+1}}(\widetilde b_{l+1})\dots x_{\alpha_m}(\widetilde b_m)=\\
=t\widetilde u\mathbf w 
 x_{\alpha_1}(b_1)\dots  x_{\alpha_i}(b_i-1)\dots x_{\alpha_l} (b_l)x_{\alpha_{l+1}}(\widetilde b_{l+1}')\dots x_{\alpha_m}(\widetilde b_m').
\end{multline*}
To bring this last expression into the usual form of the standard Bruhat decomposition, we still need to move  all  the unipotents that correspond to the roots which do not change their signs, to the left of~$\mathbf w$. But when we do it, the parameter in $x_\beta(\cdot)$ does not change under the action of~$w$. Therefore we have two  distinct Bruhat decompositions of the same element. The difference is in  parameters of the unipotent $x_\beta(\cdot)$ from the right-hand side of~$\mathbf w$. Contradiction, hence $\mathbf w(\beta)\in \Phi^+$.
\end{proof}

\begin{lemma}
For any root $\beta\in \{ \alpha_3,\dots, \alpha_l\}$ we have $\mathbf w(\beta)=\beta$.
\end{lemma}

\begin{proof}
If $\beta\in \{ \alpha_3,\dots, \alpha_l\}$, then $x_\beta(1), x_{-\beta}(1)\in \Gamma_\alpha$, therefore $w_\beta(1)$ commutes with~$g$. Also let us mention that from the right side of~$\mathbf w$ in the Bruhat decomposition of~$g$ the element $x_\beta(\cdot)$ is omitted, since $\mathbf w(\beta)\in \Phi^+$ (by the previous lemma).

Since $\beta$ is a simple root, $\mathbf w_\beta$ maps all positive roots (except~$\beta$) also to  positive roots.

Therefore,
\begin{multline*}
tx_{\alpha_1}(a_1)\dots x_\beta(a_\beta)\dots  x_{\alpha_m}(a_m) \mathbf w x_{\alpha_1}(b_1)\dots x_{\alpha_m}(b_m)=g
=g^{w_\beta(1)}=\\
=\widetilde tx_{\mathbf w_\beta(\alpha_1)}( a_1)\dots x_{\mathbf w_\beta(\beta)}( a_\beta)\dots x_{\mathbf w_\beta(\alpha_m)}(a_m)\mathbf w^{w_\beta(1)} x_{\mathbf w_\beta(\alpha_1)}(b_1)\dots   x_{\mathbf w_\beta(\alpha_m)}( b_m)=\\
=\widetilde tx_{\alpha_1}(\widetilde a_1)\dots  x_{-\beta} (a_\beta)\dots x_{\alpha_m}(\widetilde  a_m)\mathbf w^{w_\beta(1)}  x_{\mathbf w_\beta(\alpha_1)}(b_1)\dots   x_{\mathbf w_\beta(\alpha_m)}( b_m)=\\
\text{(since for any $\delta\in \Phi^+$ $[x_\delta(a),x_{-\beta}(a_\beta)]$ does not contain any unipotents with negative roots)}\\
=\widetilde tx_{\alpha_1}(\widetilde a_1)\dots x_{\alpha_m}(\widetilde  a_m) x_{-\beta} (a_\beta)\mathbf w^{w_\beta(1)}  x_{\mathbf w_\beta(\alpha_1)}(b_1)\dots   x_{\mathbf w_\beta(\alpha_m)}( b_m)=\\
=\widetilde tx_{\alpha_1}(\widetilde a_1)\dots x_{\alpha_m}(\widetilde a_m)w_\beta(1)x_{\beta}(a_\beta)\mathbf w  w_\beta(1)^{-1} x_{\mathbf w_\beta(\alpha_1)}(b_1)\dots   x_{\mathbf w_\beta(\alpha_m)}( b_m)=\\
=\widetilde tx_{\alpha_1}(\widetilde a_1)\dots x_{\alpha_m}(\widetilde a_m)w_\beta(1)\mathbf w x_{\mathbf w(\beta)}(a_\beta)  w_\beta(1)^{-1} x_{\mathbf w_\beta(\alpha_1)}(b_1)\dots   x_{\mathbf w_\beta(\alpha_m)}( b_m)=\\
=\widetilde tx_{\alpha_1}(\widetilde a_1)\dots x_{\alpha_m}(\widetilde a_m)\mathbf w^{w_\beta(1)}  x_{\mathbf w_\beta\mathbf w(\beta)}(a_\beta)  x_{\mathbf w_\beta(\alpha_1)}(b_1)\dots   x_{\mathbf w_\beta(\alpha_m)}( b_m).
\end{multline*}
Since $\mathbf w(\beta)\in \Phi^+$, $\beta$ is a simple root, then either $\mathbf w_\beta(\mathbf w(\beta))\in \Phi^+$ or $\mathbf w (\beta)=\beta$ (in this last case our lemma is already proven). In the first case  all unipotents from the right side of~$\mathbf w$ correspond to positive roots. Hence from the uniqueness of Bruhat decomposition it follows 
$$
\mathbf w =  \mathbf w^{w_\beta(1)}\Longrightarrow \mathbf w(\beta)=\pm \beta.
$$
Since $\mathbf w(\beta)\in \Phi^+$ we have $\mathbf w(\beta)=\beta$,
what was required.
\end{proof}

 \begin{lemma}
Suppose that $g$ is in Bruhat decomposition
$$
g=t x_{\alpha_1}(a_1)\dots x_{\alpha_m}(a_m)\mathbf w x_{\alpha_1}(b_1)\dots x_{\alpha_1}(b_m)\in C(\Gamma_\alpha),
$$
where $\alpha$ is a simple root for all systems except $\mathbf C_l$, $l\geqslant 2$ and a long simple root for 
the system $\mathbf C_l$, $l\geqslant 2$. And if $\Phi=\mathbf C_l$, $\alpha$ is short, then we suppose that 
$$
\alpha_1=e_1-e_2, \alpha_2=e_2-e_3,\dots, \alpha_{l-1}=e_{l-1}-e_l, e_l=2e_l\text{ and }\alpha=e_1+e_2.
$$
In all these cases $\mathbf w=\mathbf e_W$.
\end{lemma}

\begin{proof}
Note that the roots $\alpha_3,\dots, \alpha_l, \gamma$ are linearly independent and generate in the space~$V$ of the root system~$\Phi$ a hyperspace~$V'$. The root $\alpha_1$ is orthogonal to the vectors $\{ \alpha_3,\dots, \alpha_l\}$.

Since according to the previous Lemmas~5 and~7 our $\mathbf w$ acts identically on the basis of the hyperplane~$V'$, therefore it either is an identical mapping, or is a reflection $\mathbf w_\delta$, where $\delta$ is a root orthogonal to all  $\alpha_3,\dots, \alpha_l, \gamma$. 

Let us search for such $\delta$ in different roots systems.

{\bf 1. The root system $\mathbf A_l$.} In $\mathbb R^{l+1}$, $\alpha_i=e_i-e_{i+1}$, $3\leqslant i\leqslant l$, $\gamma=e_1-e_{l+1}$.  Since all roots of~$\mathbf A_l$ have the form $e_j-e_k$, the required $\delta$ does not exist, therefore $\mathbf w=\mathbf e_W$.

{\bf 2. The root system $\mathbf B_l$, $l\geqslant 3$.} The root system is $\{ \pm e_i, \pm e_i\pm e_j\mid 1\leqslant i< j\leqslant l\}$, $\{ \alpha_3,\dots, \alpha_l\}=\{ e_3-e_4,\dots, e_{l-1}-e_l, e_l\}$ or $\{ e_1-e_2,\dots e_{l-2}-e_{l-1}\}$ depending of the length of~$\alpha$. The maximal root $\gamma$ is $e_1+e_2$. In the first case $\delta=e_1-e_2=\alpha$, in the second case $\delta=e_l=\alpha$. Both these cases are impossible, since we know that $\mathbf w(\alpha)\in \Phi^+$. 

{\bf 3. The root system $\mathbf C_l$, $l\geqslant 2$.} The root system is $\{ \pm 2e_i, \pm e_i\pm e_j\mid 1\leqslant i< j\leqslant l\}$, $\{ \alpha_3,\dots, \alpha_l\}=\{ e_3-e_4,\dots, e_{l-1}-e_l, 2e_l\}$ or $\{ e_1-e_2,\dots e_{l-2}-e_{l-1}\}$ also depending of the length of~$\alpha$. The maximal root $\gamma$ is $2e_1$. In the first case $\delta=2e_2$, in the second case $\delta=2e_l=\alpha_1$, which is impossible, since  $\mathbf w(\alpha_1)\in \Phi^+$. 

Let us suppose now that we have the same set of simple/positive roots, but $\alpha=e_1+e_2$. In this case also $\delta=2e_2$, 
$$
g=tx_{\alpha_1}(a_1)\dots x_{\alpha_m}(a_m)\mathbf w_\delta x_{e_2-e_3}(b_{e_2-e_3})\dots x_{e_2-e_l}(b_{e_2-e_l})x_{e_2+e_3}(b_{e_2+e_3})\dots x_{e_2+e_l}(b_{e_2+e_l})x_{2e_2}(b_{2e_2}),
$$
since only  positive roots $e_2-e_i$, $e_2+e_i$ and $2e_2$  are mapped to negative roots under the action of~$\mathbf w_\delta$. But for all these roots their sum with $\alpha=e_1+e_2$ is not a root, therefore $x_\alpha(1)$ commutes with this right part. Consequently, 
\begin{multline*}
g=g^{x_\alpha(1)}=(tx_\alpha(c^\alpha_t-1))x_{\alpha_1}(a_1)\dots x_{\alpha_i}(a_i')\dots x_{\alpha_m}(a_m')x_{e_1+e_2}(1)\mathbf w_{2e_2}x_{e_1+e_2}(1)^{-1} \cdot \\
\cdot x_{e_2-e_3}(b_{e_2-e_3})\dots x_{e_2-e_l}(b_{e_2-e_l})x_{e_2+e_3}(b_{e_2+e_3})\dots x_{e_2+e_l}(b_{e_2+e_l})x_{2e_2}(b_{2e_2})=\\
=(tx_\alpha(c^\alpha_t-1))x_{\alpha_1}(a_1)\dots x_{\alpha_i}(a_i')\dots x_{\alpha_m}(a_m')x_{e_1+e_2}(1)x_{e_1-e_2}(\pm 1)\mathbf w_{2e_2} \cdot \\
\cdot x_{e_2-e_3}(b_{e_2-e_3})\dots x_{e_2-e_l}(b_{e_2-e_l})x_{e_2+e_3}(b_{e_2+e_3})\dots x_{e_2+e_l}(b_{e_2+e_l})x_{2e_2}(b_{2e_2})=\\
=t x_{\alpha_1}(a_1)\dots x_{e_1-e_2}(a_{e_1-e_2}\pm 1)\dots x_{e_1+e_2}(c_t+a_{e_1+e_2})\dots x_{\alpha_i}(a_i'')\dots x_{\alpha_m}(a_m'')\mathbf w_{2e_2} \cdot \\
\cdot x_{e_2-e_3}(b_{e_2-e_3})\dots x_{e_2-e_l}(b_{e_2-e_l})x_{e_2+e_3}(b_{e_2+e_3})\dots x_{e_2+e_l}(b_{e_2+e_l})x_{2e_2}(b_{2e_2}),
\end{multline*}
since $\mathbf w_{2e_2}(e_1+e_2)=e_1-e_2$ and new $x_{e_1-e_2}(\dots)$ or $x_{e_1+e_2}(\dots)$ cannot appear from any other conjugations. Therefore this situation is impossible and $\mathbf w=\mathbf e_W$.

{\bf 4. The root system $\mathbf D_l$, $l\geqslant 4$.} The root system is $\{ \pm e_i\pm e_j\mid 1\leqslant i< j\leqslant l\}$, $\{ \alpha_3,\dots, \alpha_l\}=\{ e_3-e_4,\dots, e_{l-1}-e_l, e_{l-1}+e_l\}$, $\gamma=e_1+e_{l+1}$.  Since all roots of~$\mathbf D_l$ have the form $\pm e_j \pm e_k$, the required $\delta$ does not exist, therefore $\mathbf w=\mathbf e_W$. 

{\bf 5. The root system $\mathbf E_l$, $l=6,7,8$.} We will not present here all lists of roots, taking them from~\cite{Burbaki}. The explicit check shows that in all three cases $\delta=\alpha_1$, which is impossible, since  $\mathbf w(\alpha_1)\in \Phi^+$.

{\bf 6. The root system $\mathbf F_4$.} The roots of the system $\mathbf F_4$ are 
$$
\pm e_i, i=1,2,3,4,\quad \pm e_i\pm e_j, i,j=1,2,3,4, i\ne j,\quad \frac{1}{2}(\pm e_1\pm e_2\pm e_3\pm e_4);
$$
we suppose that 
$\Delta=\{ \frac{1}{2}(e_1-e_2-e_3-e_4), e_4, e_3-e_4,e_2-e_3.\}$
Therefore $\{ \alpha_3,\alpha_4\}=\{ e_3-e_4, e_2-e_3\}$ or $\{  \frac{1}{2}(e_1-e_2-e_3-e_4), e_4\}$, $\gamma=e_1+e_2$. In the first case $\delta=\frac{1}{2}(e_1-e_2-e_3-e_4)=\alpha_1$, in the second case such $\delta$ does not exist. The first case is also impossible, since we know that $\mathbf w(\alpha)\in \Phi^+$.

{\bf 7. The root system $\mathbf G_2$.}  If $\alpha$ is the long simple root of~$\mathbf G_2$, $\beta$ is the short simple root, then in any case $\delta$ is orthogonal to the maximal root $2\alpha+3\beta$, i.\,e. $\delta=\beta$. Certainly it is impossible if our initial $\alpha_1$ is $\beta$, so let us assume $\alpha_1=\alpha$.

Suppose that $\mathbf w=\mathbf w_\beta$, 
$$
g=t x_\alpha (a_1)x_\beta(a_2)x_{\alpha+\beta}(a_3)x_{\alpha+2\beta}(a_4)x_{\alpha+3\beta}(a_5) x_{2\alpha+3\beta}(a_6)\mathbf w_\beta x_\beta(b_2).
$$
Since $x_{\alpha+2\beta}(1)\in \Gamma_\alpha$, we can conjugate $g$ by $x_{\alpha+2\beta}(1)$:
\begin{multline*}
 g^{x_{\alpha+2\beta}(1)}=(tx_{\alpha+2\beta}(c^{\alpha+2\beta}_t-1)) x_\alpha (a_1)(x_\beta(a_2)x_{\alpha+3\beta}(\pm 2 a_2))(x_{\alpha+\beta}(a_3)x_{2\alpha+3\beta}(\pm 2a_3))\cdot\\
 \cdot x_{\alpha+2\beta}(a_4)x_{\alpha+3\beta}(a_5)
\cdot  x_{2\alpha+3\beta}(a_6) x_{\alpha+2\beta}(1) \mathbf w_\beta x_{\alpha+2\beta}(1)^{-1} x_\beta(b_2)=\\
=(tx_{\alpha+2\beta}(c^{\alpha+2\beta}_t-1)) x_\alpha (a_1)(x_\beta(a_2)x_{\alpha+3\beta}(\pm 2 a_2))(x_{\alpha+\beta}(a_3)x_{2\alpha+3\beta}(\pm 2a_3))\cdot\\
 \cdot x_{\alpha+2\beta}(a_4)x_{\alpha+3\beta}(a_5)
\cdot  x_{2\alpha+3\beta}(a_6) x_{\alpha+2\beta}(1) x_{\alpha+\beta}(\pm 1) \mathbf w_\beta  x_\beta(b_2)=\\
=t x_\alpha (a_1) x_{\beta}(a_2)x_{\alpha+\beta}(a_3\pm 1)x_{\alpha+2\beta}(\dots)x_{\alpha+3\beta}(\dots)x_{2\alpha+3\beta}(\dots) \mathbf w_\beta x_\beta(b_2),
\end{multline*}
which is impossible, since $a_3\ne a_3\pm 1$.
Therefore $\mathbf w=\mathbf e_W$.
\end{proof}

Therefore one can assume that if $g\in C_G(\Gamma_\alpha)$, then $g$ has a form 
$t x_{\alpha_1}(a_1)\dots x_{\alpha_m}(a_m)$, i.\,e., $g\in B$.

It remains to prove that  any $g\in B\cap C_G(\Gamma_\alpha)$  has a form  specified in Theorem~2.

We will prove this fact by inspection of all root systems consequently. 

\begin{observation} 
All calculations below remain valid for an arbitrary commutative ring with unity.
\end{observation}

\begin{lemma}\label{G2-fields}
If $\Phi={\mathbf G}_2$ and 
$$
g=t x_{\alpha_1}(a_1)\dots x_{\alpha_m}(a_m)\in C(\Gamma_\alpha),
$$
then $a_2=\dots =a_m=0$ and $t\in Z(G)$.
\end{lemma}

\begin{proof}
If $\alpha$ and $\beta$ are simple roots of the system $\mathbf G_2$, $\alpha$ is long and $\beta$ is short, then
(see~\cite{Steinberg}, Lemma~57):
\begin{align*}
[x_\alpha(t),x_\beta(u)]&=x_{\alpha+\beta}(tu)x_{\alpha+3\beta}(-tu^3)x_{\alpha+2\beta}(-tu^2)x_{2\alpha+3\beta}(t^2u^3),\\
[x_{\alpha+\beta}(t), x_\beta(u)]&=x_{\alpha+2\beta}(2tu)x_{\alpha+3\beta}(-3tu^2)x_{2\alpha+3\beta}(3t^2u^3),\\
[x_\alpha(t),x_{\alpha+3\beta}(u)]&=x_{2\alpha+3\beta}(tu),\\
[x_{\alpha+2\beta}(t),x_\beta(u)]&=x_{\alpha+3\beta}(-3tu),\\
[x_{\alpha+\beta}(t),x_{\alpha+2\beta}(u)]&=x_{2\alpha+3\beta}(3tu).
\end{align*}

{\bf Case 1,  $\alpha_1=\alpha$ is a long root}. Then 
$$
x_\alpha(1), x_{\alpha+\beta}(1), x_{\alpha+2\beta}(1), x_{2\alpha+3\beta}(1), x_{-\beta}(1), x_{-\alpha-2\beta}(1), x_{-\alpha-3\beta}(1)\in \Gamma_\alpha.
$$
If 
$$
g=tx_\alpha(t_1)x_\beta(t_2)x_{\alpha+\beta}(t_3)x_{\alpha+2\beta}(t_4)x_{\alpha+3\beta}(t_5)x_{2\alpha+3\beta}(t_6)\in C(\Gamma_\alpha),
$$
then 
\begin{multline*}
g^{x_\alpha(1)}=(tx_\alpha(c^\alpha_t-1))x_\alpha(t_1)(x_\beta(t_2)x_{\alpha+\beta}(t_2)x_{\alpha+3\beta}(-t_2^3)x_{\alpha+2\beta}(-t_2^2)x_{2\alpha+3\beta}(t_2^3))\cdot\\
\cdot x_{\alpha+\beta}(t_3)x_{\alpha+2\beta}(t_4) (x_{\alpha+3\beta}(t_5)x_{2\alpha+3\beta}(t_5))x_{2\alpha+3\beta}(t_6)=\\
=t x_\alpha (c^\alpha_t-1+t_1) x_\beta(t_2)x_{\alpha+\beta}(t_2+t_3)x_{\alpha+2\beta}(t_4-t_2^2)x_{\alpha+3\beta}(t_5-t_2^3)x_{2\alpha+3\beta}(t_6+t_5+t_2^3-3t_2^2t_3).
\end{multline*}
From the uniqueness of Bruhat decomposition we have
$$
c^\alpha_t=1,\quad t_2=0,\quad t_5=0,
$$ 
therefore
$$
[t,x_\alpha(1)]=1\text{ and }g=t x_\alpha(t_1)x_{\alpha+\beta}(t_3)x_{\alpha+2\beta}(t_4)x_{2\alpha+3\beta}(t_6).
$$
Since $x_{\alpha+\beta}(1)\in \Gamma_\alpha$, we have (similarly) $[t,x_{\alpha+\beta}(1)]=1$ and $3t_4=0$, therefore $t\in Z(G)$ according to Remark~\ref{Remark_torus}.

Now let us use $x_{-\beta}(1)$, taking into account that $[t,x_{_\beta}(1)]=1$ and $3t_4=0$:
\begin{multline*}
g^{x_{-\beta}(1)}=t  x_\alpha(t_1)(x_\alpha (3t_3)x_{\alpha+\beta}(t_3))(x_{\alpha+\beta}(2t_4)x_\alpha(-3t_4)x_{-\beta}(3t_4^2)x_{\alpha+2\beta}(t_4))x_{2\alpha+3\beta}(t_6)=\\
=tx_\alpha(t+3t_3)x_{\alpha+\beta}(t_3+2t_4)x_{\alpha+2\beta}(t_4)x_{2\alpha+3\beta}(t_6),
\end{multline*}
so $2t_4=0$ and together with $3t_4=0$ it gives $t_4=0$, and therefore 
$$
g=t x_\alpha(t_1)x_{\alpha+\beta}(t_3)x_{2\alpha+3\beta}(t_6),\quad 3t_3=0, t\in Z(G). 
$$
Let us now apply $x_{-\alpha-2\beta}(1)$:
\begin{multline*}
g^{x_{-\alpha-2\beta}(1)}=\\
=t  x_\alpha(t_1)(x_{-\beta}(\pm 2t_3)x_\alpha(\pm 3t_3^2)x_{-\alpha-3\beta}(\pm 3t_3)x_{\alpha+\beta}(t_3)) (x_{-\alpha-2\beta}(1)x_{2\alpha+3\beta}(t_6)x_{-\alpha-2\beta}(1)^{-1})=\\
=tx_\alpha(t_1)x_{-\beta}(\pm 2t_3)x_{\alpha+\beta}(t_3)x_{\alpha+\beta}(t_6)x_{-\beta}(-t_6)x_{-\alpha-3\beta}(t_6)x_\alpha(t_6^2)x_{2\alpha+3\beta}(t_6),
\end{multline*}
therefore $t_6=0$, then $2t_3=0$ and together with $3t_3=0$ it implies $t_3=0$.

So we proved that $g=tx_\alpha(t_1)$, where $t\in Z(G)$.

\smallskip

{\bf Case 2,  $\alpha_1=\beta$ is a short root}. Then 
$$
x_\beta(1), x_{\alpha+3\beta}(1),  x_{2\alpha+3\beta}(1), x_{-\alpha}(1), x_{-2\alpha-3\beta}(1)\in \Gamma_\beta.
$$
If 
$$
g=tx_\beta(t_2)x_\alpha(t_1)x_{\alpha+\beta}(t_3)x_{\alpha+2\beta}(t_4)x_{\alpha+3\beta}(t_5)x_{2\alpha+3\beta}(t_6)\in C(\Gamma_\beta),
$$
then 
$$
g^{x_{\alpha+3\beta}(1)}=(tx_{\alpha+3\beta}(c^{\alpha+3\beta}_t-1))x_\beta(t_2) x_\alpha(t_1)x_{2\alpha+3\beta}(t_1)
 x_{\alpha+\beta}(t_3)x_{\alpha+2\beta}(t_4) x_{\alpha+3\beta}(t_5)x_{2\alpha+3\beta}(t_6),
$$
therefore $[t,x_{\alpha+3\beta}(1)]=1$ and $t_1=0$.

Now
\begin{multline*}
g^{x_\beta(1)}=(tx_\beta(c^\beta_t-1))\cdot \\
\cdot x_\beta(t_2)(x_{\alpha+\beta}(t_3)x_{\alpha+2\beta}(2t_3)x_{\alpha+3\beta}(-3t_3)x_{2\alpha+3\beta}(3t_3^2)) (x_{\alpha+2\beta}(t_4)x_{\alpha+3\beta}(-3t_4))x_{\alpha+3\beta}(t_5)x_{2\alpha+3\beta}(t_6)=\\
=tx_\beta(c^\beta_t-1+t_2)x_{\alpha+\beta}(t_3)x_{\alpha+2\beta}(2t_3+t_4)x_{\alpha+3\beta}(-2t_3-3t_4+t_5)x_{2\alpha+3\beta}(3t_3^2+t_6),
\end{multline*}
so $[t,x_\beta(1)]=1$ (and consequently by Remark~\ref{Remark_torus} 
$t\in Z(G)$), $2t_3=3t_3^2=0$ (therefore $t_3^2=0$) and $3t_4=0$.

Let us apply $x_{-\alpha}(1)$, taking into account $t_3^2=0$:
$$
g^{x_{-\alpha}(1)}=tx_\beta(t_2)(x_{\alpha+\beta}(t_3)x_\beta(\pm t_3))x_{\alpha+2\beta}(t_4)x_{\alpha+3\beta}(t_5)x_{2\alpha+3\beta}(t_6) x_{\alpha+3\beta}(\pm t_6),
$$
consequently $t_3=t_6=0$.

Since $x_{2\alpha+3\beta}(1)$ and $x_{-2\alpha-3\beta}(1)$ are both in~$\Gamma_\beta$, then $[\mathbf w_{2\alpha+3\beta}(1),g]=1$, and since $\mathbf w_{2\alpha+3\beta}(\alpha+2\beta)=-\alpha-\beta$, $\mathbf w_{2\alpha+3\beta}(\alpha+3\beta)=-\alpha$, then $t_4=t_5=0$, what was required.
\end{proof}

\begin{lemma}\label{ADE-fields}
If $\Phi=\mathbf A_l, \mathbf D_l$ or $\mathbf E_l$, $l\geqslant 2$, 
$$
g=t x_{\alpha_1}(a_1)\dots x_{\alpha_m}(a_m)\in C(\Gamma_\alpha),
$$
then $a_2=\dots =a_m=0$ and $t\in Z(G)$.
\end{lemma}

\begin{proof}
In this simply laced case all roots has the same length and are conjugated up to the action of the Weil group, therefore we can always suppose that $\alpha=\alpha_1$.

If $\beta\ne \alpha$ is some positive root, then either it is orthogonal to~$\alpha$, or $\alpha$ and $\beta$ are simple roots of~$\mathbf A_2$ (if the angle between them is $120^\circ$), or $\alpha$ and $\gamma=\beta-\alpha$ are simple roots of~$\mathbf A_2$ (if the angle between them is $60^\circ$).

In the second and third cases we have the roots $\alpha, \beta, \alpha+\beta$, forming the system~$\mathbf A_2$, where $x_\alpha(1), x_{\alpha+\beta}(1)$ and $x_{-\beta}(1)$ belong to~$\Gamma_\alpha$.

If
$$
g=tx_\alpha(t_1)\dots x_\beta(t_\beta)\dots x_{\alpha+\beta}(t_{\alpha+\beta})\dots,
$$
then
$$
g^{x_\alpha(1)}=tx_\alpha(c^\alpha_t-1+t_1)\dots x_\beta(t_\beta)x_{\alpha+\beta}(t_\beta)\dots x_{\alpha+\beta}(t_{\alpha+\beta})\dots,
$$
where $x_\alpha(\dots), x_\beta(\dots), x_{\alpha+\beta}(\dots)$ do not appear in any other  places. Therefore $[t,x_\alpha(1)]=1$ and $t_\beta=0$.

Now 
$$
g=tx_\alpha(t_1)\dots x_{\alpha+\beta}(t_{\alpha+\beta})\dots
$$
and
$$
g^{x_{-\beta}(1)}=(tx_{-\beta}(c^{-\beta}_t-1)) x_\alpha(t_1)\dots x_\alpha(t_{\alpha+\beta})x_{\alpha+\beta}(t_{\alpha+\beta})\dots,
$$ 
where $x_\alpha(\dots), x_{-\beta}(\dots), x_{\alpha+\beta}(\dots)$ do not appear in any other  places, for all other $x_\gamma(\dots)$, except $x_{-\beta}(c^{-\beta}_t-1)$, $\gamma$ are positive. Therefore $[t,x_{-\beta}(1)]=1$ and $t_{\alpha+\beta}=0$.

We see now that if for $\gamma \ne \alpha$ an element $x_\gamma(\dots)$ appears in~$g$, then $\gamma$ is orthogonal to~$\alpha$. 
But in this case both $x_\gamma(1), x_{-\gamma}(1)\in \Gamma_\alpha$, therefore $[w_\gamma(1),g]=e$. Starting with simple roots~$\gamma$ we see that
$$
 g^{w_\gamma(1)}=\widetilde t x_\alpha(t_1) x_{-\gamma}(t_\gamma)\dots x_{\mathbf w_\gamma(\delta)}(t_\delta)\dots,
$$
where all $\mathbf w_\gamma(\delta)\in \Phi^+$. Therefore $[t,x_\gamma(1)]=1$, $t_\gamma=0$.
The element $t$ commutes with all $x_\gamma(1)$ for simple roots $\gamma$, therefore by Remark~\ref{Remark_torus} we have $t\in Z(G)$.

Also we proved that in $g$ there are no elements $x_\gamma(t_\gamma)$ with simple $\gamma$, except $x_\alpha(t_1)$. Continuing with roots of the height~$2$, we do the same and notice that for all of them also $t_\gamma=0$. Continuing this procedure to the roots of the height $3,4,\dots$, we come out with the fact that for all $\gamma\in \Phi^+\setminus \alpha$ we have $t_\gamma=0$, what was required.
\end{proof}

\begin{lemma}\label{B2-fields}
For the root system $\mathbf B_2=\{ \alpha, \beta, \alpha+\beta, \alpha+2\beta\}$ and $g=tx_\alpha(t_1)x_\beta(t_2)x_{\alpha+\beta}(t_3)x_{\alpha+2\beta}(t_4)$:

\emph{(1)} if  $g\in C(\Gamma_\alpha)$, then $g=tx_\alpha(t_1)$ with $t\in Z(G)$;

\emph{(2)}  if $g\in C(\Gamma_{\alpha+\beta})$, then $g=tx_\alpha(t_1)x_{\alpha+\beta}(t_3)x_{\alpha+2\beta}(t_4)$ with $t\in Z(G)$.
\end{lemma}

\begin{proof}
Let us remind the formula for the commutator of elementary unipotents for the simple roots in the system $\mathbf B_2$ (see \cite{Steinberg}, Lemma~33). If $\alpha$ is a long simple root, $\beta$ is a short simple root, then
\begin{align*}
[x_\alpha(t), x_\beta(u)]&=x_{\alpha+\beta}(\pm tu)x_{\alpha+2\beta}(\pm tu^2),\\
[x_{\alpha+\beta}(t),x_\beta(u)]&=x_{\alpha+2\beta}(\pm 2tu).
\end{align*}

{\bf Case 1, $g\in C(\Gamma_\alpha)$}.
In this case
$$
x_\alpha(1), x_{\alpha+\beta}(1), x_{\alpha+2\beta}(1), x_{-\beta}(1), x_{-\alpha-2\beta}(1)\in \Gamma_\alpha.
$$
Starting with $x_\alpha(1)$ we obtain
\begin{multline*}
g=tx_\alpha(t_1)x_\beta(t_2)x_{\alpha+\beta}(t_3)x_{\alpha+2\beta}(t_4)=\\
=g^{x_\alpha(1)}=(tx_\alpha(c^\alpha_t-1))x_\alpha(t_1)(x_\beta(t_2)x_{\alpha+\beta}(\pm t_2)x_{\alpha+2\beta}(\pm t_2^2)) x_{\alpha+\beta}(t_3)x_{\alpha+2\beta}(t_4),
\end{multline*}
therefore $[t,x_\alpha(1)]=1$, $t_2=0$. 

Then 
$$
g^{x_{-\beta}(1)}=(tx_{-\beta}(c_t^{-\beta}-1)) x_\alpha(t_1) (x_\alpha (\pm 2t_3)x_{\alpha+\beta}(t_3))(x_{\alpha+\beta}(\pm t_4) x_{\alpha}(\pm t_4) x_{\alpha+2\beta}(t_4)),
$$
therefore $[t,x_{-\beta}(1)]=1$ (and so $t\in Z(G)$ by Remark~\ref{Remark_torus}), $t_4=0$ and $2t_3=0$. 

Consequently, 
$$
g=tx_\alpha(t_1) x_{\alpha+\beta}(t_3),\quad t\in Z(G), 2t_3=0.
$$
Since $[w_{\alpha+2\beta}(1), g]=1$, then
$$
g=g^{w_{\alpha+2\beta}(1)}=tx_\alpha(t_1) x_{\mathbf w_{\alpha+2\beta}(\alpha+\beta)}(t_3)=tx_\alpha(t_1) x_{-\beta}(\pm t_3),
$$
therefore $t_3=0$.

The first case is complete.

\smallskip
 
{\bf Case 2, $g\in C(\Gamma_{\alpha+\beta})$}.
In this case only
$$
x_\alpha(1), x_{\alpha+\beta}(1), x_{\alpha+2\beta}(1)\in \Gamma_\beta.
$$
Starting with $x_\alpha(1)$ we obtain
$$
g^{x_\alpha(1)}=(tx_\alpha(c^\alpha_t-1))x_\alpha(t_1)(x_\beta(t_2)x_{\alpha+\beta}(\pm t_2)x_{\alpha+2\beta}(\pm t_2^2)) x_{\alpha+\beta}(t_3) x_{\alpha+2\beta}(t_4),
$$
therefore $t_2=0$, $[t,x_\alpha(1)]=1$ and $g=t x_\alpha(t_1)x_{\alpha+\beta}(t_3)x_{\alpha+2\beta}(t_4)$. 

Let us use $x_{\alpha+\beta}(1)$:
$$
g^{x_{\alpha+\beta}(1)}=(t x_{\alpha+\beta}(c^{\alpha+\beta}_t-1)) x_\alpha(t_1) x_{\alpha+\beta}(t_3)) x_{\alpha+2\beta}(t_4),
$$
therefore  $[t,x_{\alpha+\beta}(1)]=1$, and thus 
$$
g=t x_\alpha(t_1)x_{\alpha+\beta}(t_3)x_{\alpha+2\beta}(t_4),\quad t\in Z(G),
$$
what was required.
\end{proof}

\begin{lemma}\label{Lemma-BCF}
If $\Phi=\mathbf B_l, \mathbf C_l, \mathbf F_4$, $l\geqslant 3$, $g=t x_{\alpha_1}(t_1)\dots x_{\alpha_m}(t_m)$, $\alpha_1,\dots, \alpha_m\in \Phi^+$,  the roots $\alpha_1,\dots, \alpha_m$ are ordered by their heights, $g\in C(\Gamma_{\alpha})$, then:

\emph{(1)}  if  $\alpha$ is a long root or $\alpha$ is short and $\Phi=\mathbf F_4$ or $\mathbf B_l$ for $l\geqslant 3$, then $g=t x_{\alpha_1}(t_1)$, where $t\in Z(G)$;

\emph{(2)} if $\alpha=e_1+e_2$ in the standard system~$\mathbf C_l$, $l\geqslant 3$, then $g=t x_{e_1+e_2}(t_1)x_{2e_1}(t_2)x_{2e_2}(t_3)$, where $t\in Z(G)$.
\end{lemma}

\begin{proof}

{\bf Step 1.} 
Suppose that $\alpha=\alpha_1$ is a long simple root of~$\Phi$. In this case any  other positive root $\gamma$ of~$\Phi$ form together with~$\alpha$ the following configuration:

(1) $\alpha$ and $\gamma$ are orthogonal (and their sum is not a root);

(2) $\alpha$ and $\gamma$ or $\alpha$ and $\gamma-\alpha$ are simple roots of the system~$\mathbf A_2$;

(3) $\alpha$ and $\gamma$ or $\alpha$ and $\gamma-\alpha$ are simple roots of the system~$\mathbf B_2$, where $\alpha$ is long.

In the second case according to Lemma~\ref{ADE-fields} in $x_\gamma(t_\gamma)$ and $x_{\alpha+\gamma}(t_{\alpha+\gamma})$ (or $x_{\gamma-\alpha}(t_{\gamma-\alpha})$) we have $t_\gamma=t_{\gamma\pm \alpha}=0$, also $[x_\alpha(1),t]=1$, $[x_\gamma(1),t]=1$.

In the third case according to Lemma~\ref{B2-fields} in $x_\gamma(t_\gamma)$, $x_{\alpha+\gamma}(t_{\alpha+\gamma})$ and $x_{\alpha+2\gamma}(t_{\alpha+2\gamma})$ we have $t_\gamma=t_{\alpha+\gamma}=t_{\alpha+2\gamma}=0$, also $[x_\gamma(1),t]=1$.

Applying these arguments we obtain $g=tx_\alpha(t_1)\dots x_\gamma(t_\gamma)\dots$, where all $\gamma$ appeared in~$g$ are orthogonal to~$\alpha$.

After that we use the same arguments as in Lemma~\ref{ADE-fields} with conjugating~$g$ by $w_\gamma(1)$ starting from simple roots~$\gamma$ and finishing by the highest roots.

Therefore for all long roots~$\alpha$ the lemma is proved.

Now we can assume that $\alpha$ is short.

\smallskip

{\bf Step 2.} Let us look at the system~$\mathbf B_l$, $l\geqslant 3$.

The roots of this system are $\{ \pm e_i, \pm e_i\pm e_j \mid 1\leqslant i< j\leqslant l\}$, where $e_1,\dots, e_l$ is a standard basis of the Euclidean space. We can suppose in our case, that $\alpha=\alpha_1=e_l$, $\alpha_2=e_{l-1}-e_l$, \dots, $\alpha_l=e_1-e_2$. 

Let us start with conjugation by $x_\alpha(1)\in \Gamma_\alpha$:
\begin{multline*}
g^{x_\alpha(1)}=(tx_\alpha(c^\alpha_t-1))x_\alpha(t_1)\dots (x_{e_i}(t_{e_i})x_{e_i+e_l}(\pm 2t_{e_i}))\dots \\
\dots (x_{e_i-e_l}(t_{e_i-e_l})x_{e_i}(\pm t_{e_i-e_l})x_{e_i+e_l}(\pm t_{e_i-e_l}))\dots,
\end{multline*}
therefore $[x_\alpha(1),t]=1$ and $t_{e_i-e_l}=0$ for all $i=1,\dots, l-1$.

Similarly let us conjugate $g$ by $x_{e_l-e_i}(1)\in \Gamma_\alpha$:
\begin{multline*}
 g^{x_{e_l-e_i}(1)}=(tx_{e_l-e_i}(c_t^{e_l-e_i}-1))x_\alpha(t_1)\dots (x_{e_i+e_j}(t_{e_i+e_j})x_{e_j+e_l}(t_{e_i+e_j}))\dots \\
\dots (x_{e_i}(t_{e_i})x_{e_l}(\pm t_{e_i})x_{e_l+e_i}(\pm t_{e_i}^2))\dots,
\end{multline*}
therefore $[x_{e_l-e_i}(1),t]=1$ (and so $t\in Z(G)$) and $t_{e_i+e_j}=0$ for all $i\ne j$, $i,j\ne l$, and $t_{e_i}=0$ for all $i\ne l$.

Now
$$
g=t x_\alpha(t_1) \dots x_{e_i-e_j}(t_{e_i-e_j})\dots x_{e_i+e_l}(t_{e_i+e_l})\dots,\text{ where }i,j\ne l, i< j.
$$
Let us conjugate $g$ by $x_{e_j+e_l}(1)\in \Gamma_\alpha$:
$$
 g^{x_{e_j+e_l}(1)}=t x_\alpha(t_1)\dots (x_{e_i-e_j}(t_{e_i-e_j})x_{e_i+e_l}(t_{e_i-e_j}))\dots x_{e_i+e_l}(t_{e_i+e_l})\dots,
$$
therefore $t_{e_i-e_j}=0$ and
$$
g=t x_\alpha(t_1) x_{e_{l-1}+e_l}(t_{e_{l-1}+e_l})\dots x_{e_1+e_l}(t_{e_1+e_l}).
$$
Since $l\geqslant 3$, for every $i\ne l$ we have $k\ne i,l$. Taking for such~$k$ the element $x_{e_k-e_i}(1)\in \Gamma_\alpha$, we have
 $$
g^{x_{e_k-e_i}(1)}=t x_\alpha(t_1)\dots \dots (x_{e_i+e_l}(t_{e_i+e_l})x_{e_k+e_l}(t_{e_i+e_l}))\dots,
$$
therefore $t_{e_i+e_l}=0$ and $g=tx_\alpha(t_1)$, what was required.

\smallskip

{\bf Step 3.} Looking at the system~$\mathbf F_4$,  the same arguments as on the step 2 with conjugation consequently by $x_\alpha(1)$, $x_{-\alpha_2}(1)$, $x_{-\alpha_3}(1)$, $x_{-\alpha_4}(1)$, etc., yield the same result: $g=tx_\alpha(t_1)$, where $t\in Z(G)$.

\smallskip

{\bf Step 4.}
The last case is the most interesting: the root system $\mathbf C_l$, $l\geqslant 3$. The roots are
$$
\pm 2e_i\text{ and } \pm e_i \pm e_j,\text{ where } 1\leqslant i< j\leqslant l,
$$
 the set of positive roots is
$$
e_i-e_j,\ i> j,\quad e_i+e_j,\ i\ne j,\quad 2e_i.
$$
As above we suppose here that $\alpha=e_1+e_2$.

The set $\Gamma_\alpha$ consists of
$$
\{ e_1+e_2;  2e_1; 2e_2; \pm 2e_i, e_1\pm e_i, e_2\pm e_i, i\geqslant 3; \pm e_i\pm e_j, i\ne j, i,j \geqslant 3\}.
$$
Conjugating $g$ by $x_{e_1+e_i}(1)$, $i\geqslant 3$, we have
\begin{multline*}
g^{x_{e_1+e_i}(1)}=(tx_{e_1+e_i}(c^{e_1+e_i}_t-1))x_{e_1+e_2}(t_{e_1+e_2})\dots (x_{e_2-e_i}(t_{e_2-e_i})x_{e_1+e_2}(t_{e_2-e_i}))\dots\\
\dots (x_{e_k-e_i}(t_{e_k-e_i})x_{e_1+e_k}(t_{e_k-e_i})\dots,\quad 3\leqslant k< i,
\end{multline*}
and these additional $x_{e_1+e_2}(\dots)$ and $x_{e_1+e_k}(\dots)$ cannot appear from any other elements under this conjugation. Therefore $t_{e_k-e_i}=0$ for all $2\leqslant k < i\leqslant l$. Also $[t,x_{e_1+e_i}(1)]=1$. 

Conjugating $g$ by $x_{e_2+e_i}(1)$, $i\geqslant 3$, we have
\begin{multline*}
g^{x_{e_2+e_i}(1)}=(tx_{e_2+e_i}(c^{e_2+e_i}_t-1))(x_{e_1-e_2}(t_{e_1-e_2})x_{e_1+e_i}(t_{e_1-e_2}))\dots \\\dots(x_{e_1-e_i}(t_{e_1-e_i})x_{e_1+e_2}(t_{e_1-e_i}))\dots,
\end{multline*}
and these additional $x_{e_1+e_2}(\dots)$ and $x_{e_1+e_i}(\dots)$ cannot appear from any other elements under this conjugation. Therefore $t_{e_1-e_2}=0$ and $t_{e_1-e_i}=0$ for all $3\leqslant  i\leqslant l$. Also $[t,x_{e_2+e_i}(1)]=1$.

Conjugating $g$ by $x_{e_1+e_2}(1)$, we have
$$
g^{x_{e_1+e_2}(1)}=(tx_{e_1+e_2}(c^{e_1+e_2}_t-1)) x_{\alpha_1}(a_1)\dots x_{\alpha_m}(a_m).
$$
Therefore  $[t,x_{e_1+e_2}(1)]=1$. 

Conjugating $g$ by $x_{e_2-e_i}(1)$, $i\geqslant 3$, we have
\begin{multline*}
g^{x_{e_2-e_i}(1)}=(tx_{e_2-e_i}(c^{e_2-e_i}_t-1))\dots (x_{e_1+e_i}(t_{e_1+e_i})x_{e_1+e_2}(t_{e_1+e_i}))\dots (x_{2e_i}(t_{2e_i})x_{e_2+e_i}(\pm t_{2e_i})x_{2e_2}(\pm t_{2e_2}))\\
\dots(x_{e_i+e_j}(t_{e_i+e_j})x_{e_2+e_j}(t_{e_i+e_j}))\dots (x_{e_2+e_i}(t_{e_2+e_i})x_{2e_2}(\pm 2t_{e_2+e_i})),\quad j\geqslant 3, j\ne i.
\end{multline*}
 Therefore $t_{e_1+e_i}=0$ and $t_{2e_i}=0$ for $i\geqslant 3$, $t_{e_i+e_j}=0$ for all $3\leqslant  i,j\leqslant l$, $i\ne j$. Also $[t,x_{e_2-e_i}(1)]=1$. 

The same situation is with conjugating by $x_{e_1-e_i}(1)$, $i\geqslant 3$, which gives us $t_{e_2+e_i}=0$ for $i\geqslant 3$ and  $[t,x_{e_1-e_i}(1)]=1$. 

Therefore $t\in Z(G)$ and 
$$
g=t x_{e_1+e_2}(t_{e_1+e_2})x_{2e_1}(t_{2e_1})x_{2e_2}(t_{2e_2}).
$$
what was required. Direct (and easy) calculations show that such $g$ is always in $C(\Gamma_\alpha)$ for this root system.
\end{proof}

So Theorem 2 is proved.


\subsection{Double centralizers of unipotent elements in Chevalley groups over local rings}\leavevmode

In this section we will prove the same theorem for large subgroups of the Chevalley groups over local rings.
In the paper~\cite{bunina2022} we already obtained a similar result for local rings, but we supposed their that in some root systems these rings contain $1/2$ or~$1/3$. Now it is important to prove the result for all local rings, with no restrictions.

\begin{theorem}[c.f.~\cite{bunina2022}]
For any Chevalley group (or its large subgroup) $G=G_\pi(\Phi,R)$, where $\Phi$ is an irreducible root system of a rank $>1$, $R$ is a local ring, if for some $\alpha\in \Phi$ an element $g\in C_G( \Gamma_\alpha)$, then $g=c x_\alpha(t)$, where $t\in R$, $c\in Z(G)$, except the case $\Phi=\mathbf C_l$, $l\geqslant 2$, and $\alpha$ is short.

In the case $\Phi=\mathbf C_l=\{ \pm e_i\pm e_j\mid 1\leqslant i,j\leqslant l,i\ne j\}\cup \{ \pm 2e_i\mid 1\leqslant i\leqslant l\}$ and $\alpha=e_1+e_2$ if $g\in C_G(\Gamma_\alpha)$, then 
$$
g=c x_{e_1+e_2}(t_1)x_{2e_1}(t_2)x_{2e_2}(t_3),\quad c\in Z(G).
$$
\end{theorem}

\begin{proof}
Let $\Gamma_\alpha=\{ x_\beta(1)\mid [x_\beta(1),x_\alpha(1)]=e\}$, $g\in C_G (\Gamma_\alpha)$. 

Let us denote the residue field $R/\Rad R$ by~$k$ and the images of the elements $g, x_\beta(1), t, u$ etc. in the quotient group $\overline G=G_\pi(\Phi, k)$ by $\overline g, x_\beta(\overline 1), \overline t, \overline u$, respectively.

Since $g\in C_G(\Gamma_\alpha)$, we have $\overline g\in C_{\overline G}(\overline \Gamma_\alpha)$, where $\overline \Gamma_\alpha=\{ x_\beta(\overline 1)\mid [x_\beta(\overline 1),x_\alpha(\overline 1)=\overline e\}$. Since $\overline G$ is the Chevalley group over a field, by Theorem~2  
$$
\overline g=\overline c x_\alpha (\overline t),\ \overline c\in Z(\overline G)\quad (\text{or }\overline c x_{e_1+e_2}(\overline t_1)x_{2e_1}(\overline t_2)x_{2e_2}(\overline t_3),\text{ if }\Phi=\mathbf C_l,\ l\geqslant 2,\ \alpha=e_1+e_2).
$$

As in the previous section we assume that $\alpha=\alpha_1$ is the first simple root.

In a Chevalley group $G$ over a local ring there exists a Gauss decomposition of the form $G=UTVU$ (see~\cite{Smolensky-Sury-Vavilov}, \cite{Abe1}, \cite{AS},~\cite{Iwahori-Matsumoto}). So we fix a representation of~$g$ as
$$
 g = x_{\alpha_{1}}(r_{1})\ldots x_{\alpha_{m}}(r_{m}) t x_{-\alpha_{1}}(s_{1})\ldots x_{-\alpha_{m}}(s_{m})x_{\alpha_{1}}(t_{1})\ldots x_{\alpha_{m}}(t_{m}),
$$
where  $\alpha_{i}$ are positive roots, $r_{i}, s_{i},t_i \in R$, $t\in T(R)$. Since the image of~$g$ under canonical homomorphism is
$$
 x_\alpha (\overline r_1) \overline t\quad (\text{or }x_{e_1+e_2}(\overline r_1)x_{2e_1}(\overline r_2)x_{2e_2}(\overline r_3)\overline t\text{ in the case }\Phi=\mathbf C_l, l\geqslant 2, \alpha=e_1+e_2).
$$
Therefore in all cases the elements $s_1,\dots, s_m, t_1,\dots, t_m\in \Rad R$.

According to the formula
$$
x_{-\gamma}(s) x_{\gamma} (t)=h_\gamma \left(\frac{1}{1+st}\right) x_\gamma (t(1+st)) x_{-\gamma}\left(\frac{s}{1+st}\right)\text{ for all }\gamma\in \Phi\eqno (*)
$$
(it is checked directly through a representation of $x_\gamma(\cdot)$, $x_{-\gamma}(\cdot)$, $h_\gamma(\cdot)$ by matrices from $\SL_2$). Since $1+st\in R^*$ if $s\in \Rad R$ or $t\in \Rad R$, we can move all $x_\beta(t_i)$ from the right side of decomposition to its left side. So we obtain a representation of~$g$ of the form
$$
g=x_{\alpha_{1}}(r_{1})\ldots x_{\alpha_{m}}(r_{m}) t x_{-\alpha_{1}}(s_{1})\ldots x_{-\alpha_{m}}(s_m),\quad s_1,\dots, s_m\in \Rad R.
$$
Note that, unlike the original Gauss decomposition, such a decomposition is uniquely defined. Indeed, suppose that
$$ 
u_{1}t_{1}v_{1} = u_{2}t_{2}v_{2}.
$$
If we move all the positive roots to one side:
$t_{1}v_{1}v_{2}^{-1}t_{2}^{-1} = u_{1}^{-1}u_{2}$,  then since
$TV \;\cap\; U = 1$, therefore $u_1=u_2$ and $v_1v_2^{-1}=t_2t_1^{-1}$. Since $T\cap V=1$, we have $v_1=v_2$, $t_1=t_2$, so this form of decomposition is unique.

Let us use also the formula
$$
x_\gamma(1) x_{-\gamma}(s) x_{\gamma} (1)^{-1}=h_\gamma \left(\frac{1}{1-s}\right) x_\gamma (s^2-s) x_{-\gamma}\left(\frac{s}{1-s}\right)\text{ for all }\gamma\in \Phi\eqno (**)
$$
which follows directly from~$(*)$.

As in previous section we suppose that the roots $\alpha_1,\dots, \alpha_m$ are ordered by their heights.

In our case of Gauss decomposition $UTV$ it is convenient  to suppose that for $g=utv\in UTV$ always
$$
u=x_{\alpha_1}(r_1)\dots x_{\alpha_m}(r_m)tx_{-\alpha_m}(s_m)\dots x_{-\alpha_1}(s_1).
$$

{\bf Case 1. Root systems with the roots of the same length: $\mathbf A_l$, $\mathbf D_l$, $\mathbf E_l$}.

Consider a root $\beta \in \Phi$ such that  $\alpha + \beta \notin \Phi$, and consequently the element $x_{\beta}(1)\in \Gamma_\alpha$. Since $g \in C(\Gamma_\alpha) $, then $g^{x_{\beta}(1)} = g$. Let us consider, how conjugation by this element acts on~$g$ and its  factors (assume that $\beta$ is positive):
\begin{multline*}
g=g^{x_{\beta}(1)} = x_{\alpha_{1}}(r_{1})^{x_{\beta}(1)}\ldots x_{\alpha_{m}}(r_{m})^{x_{\beta}(1)}t^{x_{\beta}(1)}x_{-\alpha_{m}}(s_{m})^{x_{\beta}(1)}\ldots x_{-\alpha_{1}}(s_{1})^{x_{\beta}(1)}=\\
= x_{\alpha_{1}}(r_{1}) \dots ( x_{\alpha_i}(r_i)x_{\alpha_i+\beta}(\pm r_i)) \dots x_{\alpha_m}(r_m) \cdot (x_\beta(c_t^\beta-1)  t)\cdot\\
\cdot  (x_{-\alpha_m}(s_m)x_{-\alpha_m+\beta}(\pm s_m)) \dots \left(h_\beta\left( \frac{1}{1-s_{-\beta}}\right)x_\beta(s_{-\beta}^2-s_{-\beta})x_{-\beta}\left( \frac{s_{-\beta}}{1-s_{-\beta}}\right)\right)
\dots x_{-\alpha_1}(s_1).
\end{multline*}
Let us  analyze the  obtained equality.

Note that from the left-hand side of $t$ there are only unipotent elements with positive roots, that is, an element of~$U$. From the right side of~$t$ there is an element  $h_\beta\left( \frac{1}{1-s_{-\beta}}\right)$ of the torus~$T$. Since torus obviously normalizes any $X_\alpha$, we can move this element to the left and obtain $t\cdot h_\beta\left( \frac{1}{1-s_{-\beta}}\right)$ instead of~$t$. 
Since all unipotents with positive roots which can appear by conjugation of any unipotent with a negative root~$-\gamma$ by~$x_{\beta}(1)$, have heights strictly smaller than~$\gamma$, then  one can move them to the left towards $T$ and $U$, and it is impossible during this movement to meet unipotents with opposite roots. This means that it is possible to move all unipotents with positive roots that are located to the right of~$t$, to the left of~$t$. Therefore $t$ and $h_\beta\left( \frac{1}{1-s_{-\beta}}\right)$ will not be changed. After that, $g^{x_\beta(1)}$ will be written in the form of $UTV$, that is
$$
t=t\cdot h_\beta\left( \frac{1}{1-s_{-\beta}}\right), 
$$
hence $s_{-\beta}=0$.

Therefore any new $x_\beta(\cdot)$ cannot appear from the left-hand or right-hand  part of~$t$ in $g^{x_\beta(1)}$.  Consequently, in the expression $t^{x_\beta(1)}=x_\beta(c_t^\beta-1)$ we necessarily have  $c_t^\beta=1$, that is
$$
[t,x_\beta(1)]=1.
$$

If $\beta\in \Gamma_\alpha\cap \Phi^+$ and $\gamma=\beta-\alpha\in \Phi$, then
\begin{multline*}
g=g^{x_{\beta}(1)} 
=u'tx_{-\alpha_m}(s_m)^{x_\beta(1)}\dots x_{-\alpha_i}(s_i)^{x_\beta(1)}\dots x_{-\gamma}(s_\gamma)^{x_\beta(1)}\dots x_{-\alpha_1}(s_1)=\\
=u't x_{-\alpha_m}(s_m)x_{-\alpha_m+\beta}(\pm s_{m})\dots x_{-\alpha_i}(s_i)x_{-\alpha_i+\beta}(\pm s_i)\dots x_{-\gamma}(s_{-\gamma})x_\alpha(\pm s_{-\gamma})\dots x_{-\alpha_1}(s_1).
\end{multline*}
In the part $V$ a new element $x_\alpha(\pm s_{-\gamma})$ appeared, and it is impossible to obtain $x_\alpha(\cdot)$ from any other conjugation $x_\delta(\cdot)^{x_\beta(1)}$. When we move $x_\alpha(\pm s_{-\gamma})$ and other $x_{-\delta+\beta}(\cdot)$, $-\delta+\beta\in \Phi^+$, to the left side, we also cannot obtain any new $x_\alpha(\cdot)$. Therefore $s_{-\gamma}=0$.

So we see that $s_{-\gamma}=0$ for all roots $\gamma\in \Phi^+$ such that $\alpha+\gamma\notin \Phi$, and for all roots $\gamma\in \Phi^+$ such that $\alpha+\gamma\in \Phi$, but $\alpha+2\gamma \notin \Phi$. But for simply laced root systems all positive roots have one of these properties. 

Therefore $g\in UT$ and we come to the situation of Lemma~\ref{ADE-fields}, which was proved for arbitrary commutative rings.
By this lemma $g=cx_{\alpha_1}(r_1)$, where $c\in Z(G)$. Consequently, Theorem~3 is proved for the root systems $\mathbf A_l, \mathbf D_l, \mathbf E_l$, $l\geqslant 2$.

\smallskip

{\bf Case 2. The root system $\mathbf G_2$}. 

Let us repeat that this root system has simple roots $\alpha,\beta$, positive roots $\alpha,\beta, \alpha+\beta, \alpha+2\beta, \alpha+3\beta, 2\alpha+3\beta$,
$$
\Gamma_\alpha=\{ x_\alpha(1), x_{\alpha+\beta}(1), x_{\alpha+2\beta}(1), x_{2\alpha+3\beta}(1), x_{-\beta}(1), x_{-\alpha-2\beta}(1), x_{-\alpha-3\beta}(1)\}
$$
and 
$$
\Gamma_\beta=\{ x_\beta(1), x_{\alpha+3\beta}(1), x_{2\alpha+3\beta}(1), x_{-\alpha}(1), x_{-2\alpha-3\beta}(1)\}.
$$
{\bf In the first case} $g\in C(\Gamma_\alpha)$ by the same arguments as in the previous case
$$
s_{-\alpha}=s_{-\alpha-\beta}=s_{-\alpha-2\beta}=s_{-2\alpha-3\beta}=r_{\beta}=r_{\alpha+2\beta}=r_{\alpha+3\beta}=0\text{ and } t\in Z(G),
$$
therefore 
$$
g=x_\alpha(r_\alpha)x_{\alpha+\beta}(r_{\alpha+\beta})x_{2\alpha+3\beta}(r_{2\alpha+3\beta})t x_{-\beta}(s_{-\beta})x_{-\alpha-3\beta}(s_{-\alpha-3\beta}),\quad t\in Z(G).
$$
Conjugating $g$ by $x_{2\alpha+3\beta}(1)$, we have
$$
g^{x_{2\alpha+3\beta}(1)}=x_\alpha(r_\alpha)x_{\alpha+\beta}(r_{\alpha+\beta})x_{2\alpha+3\beta}(r_{2\alpha+3\beta})t x_{-\beta}(s_{-\beta})x_{-\alpha-3\beta}(s_{-\alpha-3\beta})x_{\alpha}(\pm s_{-\alpha-3\beta}),
$$
therefore $s_{-\alpha-3\beta}=0$ and 
$$
g=x_\alpha(r_\alpha)x_{\alpha+\beta}(r_{\alpha+\beta})x_{2\alpha+3\beta}(r_{2\alpha+3\beta})t x_{-\beta}(s_{-\beta}),\quad t\in Z(G).
$$
Conjugating $g$ by $x_{-\beta}(1)$, we have
$$
g^{x_{-\beta}(1)}=x_\alpha(r_\alpha)x_{\alpha}(\pm 3r_{\alpha+\beta})x_{\alpha+\beta}(r_{\alpha+\beta})x_{2\alpha+3\beta}(r_{2\alpha+3\beta})t x_{-\beta}(s_{-\beta}),
$$
therefore $3r_{\alpha+\beta}=0$.

Conjugating $g$ by $x_{\alpha+2\beta}(1)$, we have
\begin{multline*}
g^{x_{\alpha+2\beta}(1)}=\\
=x_\alpha(r_\alpha)x_{\alpha+\beta}(r_{\alpha+\beta})x_{2\alpha+3\beta}(r_{2\alpha+3\beta}\pm 3r_{\alpha+\beta})t  x_{2\alpha+3\beta}(\pm 3s_{-\beta})x_{\alpha}(\pm 3 s_{-\beta}^2)x_{\alpha+\beta}(\pm 2 s_{-\beta}) x_{-\beta}(s_{-\beta})=\\
=x_\alpha(r_\alpha\pm 3s_{-\beta}^2)x_{\alpha+\beta}(r_{\alpha+\beta}\pm 2s_{-\beta})x_{2\alpha+3\beta}(r_{2\alpha+3\beta}\pm 3 r_{\alpha+\beta}\pm 3s_{-\beta}) tx_{_\beta}(s_{-\beta}),
\end{multline*}
since $3r_{\alpha+\beta}=0$, we have $2s_{-\beta}=3s_{-\beta}=0$, so $s_{-\beta}=0$.

Now $g\in UT$ and our result follows from Lemma~\ref{G2-fields}. 

\smallskip

{\bf In the second case} $g\in C(\Gamma_\beta)$ we have 
$$
s_{-\beta}=s_{-\alpha-3\beta}=s_{-2\alpha-3\beta}=r_\alpha=r_{2\alpha+3\beta}=0\text{ and }t\in Z(G),
$$
therefore
$$
g=x_\beta(r_\beta)x_{\alpha+\beta}(r_{\alpha+\beta})x_{\alpha+2\beta}(r_{\alpha+2\beta})x_{\alpha+3\beta}(r_{\alpha+3\beta})tx_{-\alpha}(s_{-\alpha})x_{-\alpha-\beta}(s_{-\alpha-\beta})x_{-\alpha-2\beta}(s_{-\alpha-2\beta}).
$$
Conjugating $g$ by $x_{-\alpha}(1)$, we have
\begin{multline*}
g^{x_{-\alpha}(1)}=x_\beta(r_\beta)(x_{\alpha+\beta}(r_{\alpha+\beta})x_\beta(\pm r_{\alpha+\beta})x_{\alpha+2\beta}(\pm r_{\alpha+\beta}^2) x_{2\alpha+3\beta}(\pm r_{\alpha+\beta}^3)x_{\alpha+3\beta}(r_{\alpha+3\beta}^3))\cdot\\
\cdot x_{\alpha+2\beta}(r_{\alpha+2\beta})x_{\alpha+3\beta}(r_{\alpha+3\beta})t x_{-\alpha}(s_{-\alpha})x_{-\alpha-\beta}(s_{-\alpha-\beta})x_{-\alpha-2\beta}(s_{-\alpha-2\beta}),
\end{multline*}
therefore $r_{\alpha+\beta}=0$.

In a similar way if we conjugate $g$ by $x_{\alpha+3\beta}(1)$, we obtain $s_{-\alpha-2\beta}=0$.

Therefore on this stage
$$
g=x_\beta(r_\beta)x_{\alpha+2\beta}(r_{\alpha+2\beta})x_{\alpha+3\beta}(r_{\alpha+3\beta})tx_{-\alpha}(s_{-\alpha})x_{-\alpha-\beta}(s_{-\alpha-\beta}).
$$
Conjugating now  $g$ by $x_{2\alpha+3\beta}(1)$, we have
\begin{multline*}
g^{x_{2\alpha+3\beta}(1)}=x_\beta(r_\beta)x_{\alpha+2\beta}(r_{\alpha+2\beta})x_{\alpha+3\beta}(r_{\alpha+3\beta})t (x_{\alpha+3\beta}(\pm s_{-\alpha})x_{-\alpha}(s_{-\alpha}))\cdot\\
\cdot (x_{-\alpha-\beta}(s_{-\alpha-\beta})x_{\alpha+2\beta}(\pm s_{-\alpha-\beta})x_\beta(\pm s_{-\alpha-\beta}^2)x_{-\alpha}(\pm s_{-\alpha-\beta}^3)x_{\alpha+3\beta}(\pm s_{-\alpha-\beta}^3)),
\end{multline*}
therefore $s_{-\alpha}=s_{-\alpha-\beta}=0$ and $g=x_\beta(r_\beta)x_{\alpha+2\beta}(r_{\alpha+2\beta})x_{\alpha+3\beta}(r_{\alpha+3\beta})t$.

Now again $g\in UT$ and the result if Theorem~3 follows from Lemma~\ref{G2-fields}.

\medskip

{\bf Case 3. The root systems $\mathbf B_l$, $l\geqslant 3$, and $\mathbf F_4$}. 

We remember that $\mathbf B_l$ consists of the roots $\{ \pm e_i, \pm e_i\pm e_j\mid 1\leqslant i,j\leqslant l, i\ne j\}$. If $\alpha$ is {\bf a long root} $e_1-e_2$, then 
$$
\Gamma_\alpha=\{ x_{e_1\pm e_2}(1), x_{-e_1-e_2}(1), x_{e_1\pm e_i}(1), x_{-e_2\pm e_i}(1), x_{\pm e_i\pm e_j}(1), x_{e_1}(1), x_{-e_2}(1), x_{\pm e_i}(1)\},
$$
where $3\leqslant i,j\leqslant l$, $i\ne j$.

Therefore by the same reasons as above
$$
s_{-e_1\pm e_2}=r_{e_1+e_2}=s_{-e_1\pm e_i}=r_{e_2\pm e_i}=r_{\pm e_i\pm e_j}=s_{-e_1}=r_{e_2}=r_{\pm e_i}=0, \qquad t\in Z(G),
$$
where $3\leqslant i,j\leqslant l$, $i\ne j$.

Therefore
\begin{multline*}
g=x_{e_1-e_2}(r_{e_1-e_2})\dots x_{e_1-e_l}(r_{e_1-e_l})x_{e_1+e_3}(r_{e_1+e_3})\dots x_{e_1+e_l}(r_{e_1+e_l})x_{e_1}(r_{e_1})t\cdot \\
\cdot x_{-e_2-e_3}(s_{-e_2-e_3})\dots x_{-e_2-e_l}(s_{-e_2-e_l})x_{-e_2+e_3}(s_{-e_2+e_3})\dots x_{-e_2+e_l}(s_{-e_2+e_l})x_{-e_2}(s_{-e_2}).
\end{multline*}

Conjugating $g$ by $x_{-e_2+e_i}(1)$, we have
\begin{multline*}
g^{x_{-e_2+e_i}(1)}=x_{e_1-e_2}(r_{e_1-e_2})\dots (x_{e_1-e_i}(r_{e_1-e_i})x_{e_1-e_2}(\pm r_{e_1-e_i}))\dots\\
\dots x_{e_1-e_l}(r_{e_1-e_l})x_{e_1+e_3}(r_{e_1+e_3})\dots x_{e_1+e_l}(r_{e_1+e_l})x_{e_1}(r_{e_1})t\cdot \\
\cdot x_{-e_2-e_3}(s_{-e_2-e_3})\dots x_{-e_2-e_l}(s_{-e_2-e_l})x_{-e_2+e_3}(s_{-e_2+e_3})\dots x_{-e_2+e_l}(s_{-e_2+e_l})x_{-e_2}(s_{-e_2}),
\end{multline*}
therefore for all $3\leqslant i\leqslant l$ we have $r_{e_1-e_i}=0$. 

Conjugating $g$ by $x_{-e_2-e_i}(1)$, we similarly obtain $r_{e_1+e_i}=0$ for all $3\leqslant i\leqslant l$.

So we see that 
\begin{multline*}
g=x_{e_1-e_2}(r_{e_1-e_2})x_{e_1}(r_{e_1})t\cdot \\
\cdot x_{-e_2-e_3}(s_{-e_2-e_3})\dots x_{-e_2-e_l}(s_{-e_2-e_l})x_{-e_2+e_3}(s_{-e_2+e_3})\dots x_{-e_2+e_l}(s_{-e_2+e_l})x_{-e_2}(s_{-e_2}).
\end{multline*}
The case when we conjugate  $g$ by $x_{e_1+e_i}(1)$ and then by $x_{e_1-e_i}(1)$ for $3\leqslant i \leqslant l$ is treated in the very similar way. We drop the corresponding calculations.  We obtain
$$
g=x_{e_1-e_2}(r_{e_1-e_2})x_{e_1}(r_{e_1})tx_{-e_2}(s_{-e_2}).
$$
Conjugating $g$ by $x_{-e_1-e_2}(1)$, we obtain $r_{e_1}=0$, and conjugating $g$  by $x_{e_1+e_2}(1)$, we obtain $s_{-e_2}=0$, what was required.

\smallskip

It was the case $\mathbf B_l$, $l\geqslant$, where $\alpha$ is a long root. Now let us suppose that $\alpha$ {\bf is short}, for example, $\alpha=e_1$.

In this case
$$
\Gamma_\alpha=\{ x_{e_1}(1), x_{e_1\pm e_i}(1), x_{\pm e_i\pm e_j}(1)\},\text{ where }2\leqslant i,j\leqslant l, i\ne j.
$$
Therefore
$$
s_{-e_1}=s_{-e_1\pm e_i}=r_{\pm e_i\pm e_j}=0\text{ for }2\leqslant i,j\leqslant l, i\ne j
$$
and 
\begin{multline*}
g=x_{e_1}(r_{e_1})\dots x_{e_l}(r_{e_l})x_{e_1-e_2}(r_{e_1-e_2})\dots x_{e_1-e_l}(r_{e_1-e_l})x_{e_1+e_2}(r_{e_1+e_2})\dots x_{e_1+e_l}(r_{e_1+e_l})t\cdot \\
\cdot x_{-e_2}(s_{-e_2})\dots x_{-e_l}(s_{-e_l}).
\end{multline*}

Conjugating $g$ by $x_{e_2-e_i}(1)$, we obtain
\begin{multline*}
g^{x_{e_2-e_i}(1)}=x_{e_1}(r_{e_1})\dots (x_{e_i}(r_{e_i}) x_{e_2}(\pm r_{e_i}) x_{e_2+e_i}(\pm r_{e_i})^2)\dots x_{e_l}(r_{e_l})\cdot\\
\cdot (x_{e_1-e_2}(r_{e_1-e_2})x_{e_1-e_i}(\pm r_{e_1-e_2}))\dots x_{e_1-e_l}(r_{e_1-e_l})x_{e_1+e_2}(r_{e_1+e_2})\dots\\
\dots  (x_{e_1+e_i}(r_{e_1+e_i})x_{e_1+e_2}(\pm r_{e_1+e_i}))\dots x_{e_1+e_l}(r_{e_1+e_l})t\cdot \\
\cdot (x_{-e_2}(s_{-e_2})x_{-e_i}(\pm s_{-e_2})x_{-e_2-e_i}(\pm s_{-e_2}^2))\dots x_{-e_l}(s_{-e_l}),
\end{multline*}
which directly implies 
$$
r_{e_i}=r_{e_1-e_2}=r_{e_1+e_i}=s_{-e_2}=0\text{ for all }3\leqslant i \leqslant l.
$$
Conjugating $g$ by $x_{e_2+e_i}(1)$ we similarly obtain 
$$
r_{e_1-e_i}=s_{-e_i}=0\text{ for all }3\leqslant i \leqslant l.
$$
 Therefore
$$
g=x_{e_1}(r_{e_1}) x_{e_2}(r_{e_2})x_{e_1+e_2}(r_{e_1+e_2})t.
$$
Conjugating $g$ now by $x_{e_3-e_2}(1)$, we obtain
$$
g^{x_{e_3-e_2}(1)}=x_{e_1}(r_{e_1}) ( x_{e_2}(r_{e_2})x_{e_3}(\pm r_{e_2})x_{e_2+e_3}(\pm r_{e_2}^2)) \cdot (x_{e_1+e_2}(r_{e_1+e_2})x_{e_1+e_3}(\pm r_{e_1+e_2}))\cdot t,
$$
therefore 
$g=x_{e_1}(r_{e_1})t$, $t\in Z(G)$, what was required.

\smallskip

{\bf The case} $\Phi=\mathbf F_4$ is treated in the very similar way. We drop the corresponding calculations

\medskip

{\bf 4. The case $\Phi = \mathbf C_l$, $l\geqslant 2$, $\alpha$ is long}.
Let (as above)
$$
\Phi=\{ \pm e_i\pm e_j, \pm 2e_i\mid 1\leqslant i,j \leqslant l, i\ne j\},\quad \alpha = 2e_1.
$$
Then 
$$
\Gamma_\alpha=\{ x_{2e_1}(1), x_{\pm 2e_i}(1), x_{e_1\pm e_i}(1), x_{\pm e_i\pm e_j}(1)\},\text{ where } 2\leqslant i,j \leqslant l, i\ne j.
$$
Therefore in $g$
$$
s_{-2e_1}=r_{2e_i}=s_{-2e_i}=s_{-e_1\pm e_i}=r_{\pm e_i\pm e_j}=0,\quad t\in Z(G)
$$
and
$$
g=x_{2e_1}(r_{2e_1})x_{e_1-e_2}(r_{e_1-e_2})\dots x_{e_1-e_l}(r_{e_1-e_l})x_{e_1+e_2}(r_{e_1+e_2})\dots x_{e_1+e_l}(r_{e_1+e_l})t,
$$
i.\,e. $g$ is the same as in Lemma~\ref{Lemma-BCF} and by the same argument $g=x_{2e_1}(r_{2e_1})t$, $t\in Z(G)$.

\medskip

{\bf 5. The case $\Phi = \mathbf C_l$, $l\geqslant 2$, $\alpha$ is short}.

We suppose that $\alpha=e_1+e_2$.
Then 
$$
\Gamma_\alpha=\{ x_{2e_1}(1), x_{2e_2}(1), x_{\pm 2e_i}(1), x_{e_1\pm e_i}(1), x_{e_2\pm e_i}(1), x_{\pm e_i\pm e_j}(1)\},\text{ where } 3\leqslant i,j \leqslant l, i\ne j,
$$
therefore in $g$
$$
s_{-2e_1}=s_{-2e_2}=r_{2e_i}=s_{-2e_i}=s_{-e_1\pm e_i}=s_{-e_2\pm e_i}=r_{\pm e_i\pm e_j}=0,\quad t\in Z(G),
$$
therefore $g\in UT$ and the same Lemma~\ref{Lemma-BCF} completes the proof.
\end{proof}

\subsection{Double centralizers of unipotent elements in Chevalley groups over arbitrary commutative rings}\leavevmode

Finally we are able to prove the same theorem for arbitrary commutative rings with unity.

\begin{theorem}\label{double_centr_arbitrary}
For any Chevalley group (or its large subgroup) $G=G_\pi(\Phi,R)$, where $\Phi$ is an irreducible root system of a rank $>1$, $R$ is an arbitrary commutative ring with~$1$, if for some $\alpha\in \Phi$ an element $g\in C_G(\Gamma_\alpha)$, then $g=c x_\alpha(t)$, where $t\in R$, $c\in Z(G)$, except the case $\Phi=\mathbf C_l$, $l\geqslant 2$, and $\alpha$ is short.

In the case $\Phi=\mathbf C_l=\{ \pm e_i\pm e_j\mid 1\leqslant i,j\leqslant l,i\ne j\}\cup \{ \pm 2e_i\mid 1\leqslant i\leqslant l\}$ and $\alpha=e_1+e_2$ if $g\in C_G(\Gamma_\alpha)$, then 
$$
g=c x_{e_1+e_2}(t_1)x_{2e_1}(t_2)x_{2e_2}(t_3),\quad c\in Z(G).
$$

\end{theorem}

\begin{proof}
We embed our ring $R$ in the Cartesian product of all its localizations by maximal ideals:
$$
R\subset \widetilde R=\prod_{\mathfrak m\text{ is a maximal ideal of }R} R_{\mathfrak m}.
$$
Respectively the Chevalley group $G=G_\pi(\Phi,R)$ is naturally embedded into the Chevalley group
$$
\widetilde G=G_\pi(\Phi,\widetilde R)=\prod_{\mathfrak m\text{ is a maximal ideal of }R} G_\pi(\Phi, R_{\mathfrak m}).
$$

Suppose that some $g\in G$ commute with all elements of the set~$\Gamma_\alpha$. Since $G\subset \widetilde G$, we have $g\in \prod_{\mathfrak m} G_\pi(\Phi, R_{\mathfrak m})$ and we can represent~$g$ as $g=(g_{\mathfrak m})_{\mathfrak m\in \mathfrak M}$, $g_{\mathfrak m}\in G_\pi(\Phi,R_{\mathfrak m})$, where $\mathfrak M$ is the set of all maximal ideals of~$R$. 

Since for any different maximal ideals $\mathfrak m_1$ and $\mathfrak m_2$ if $x= (x_{\mathfrak m})_{\mathfrak m\in \mathfrak M}$, where $x_{\mathfrak m}=e_{R_{\mathfrak m}}$ for all $\mathfrak m \ne \mathfrak m_1$   and $y= (y_{\mathfrak m})_{\mathfrak m\in \mathfrak M}$, where   $y_{\mathfrak m}=e_{R_{\mathfrak m}}$ for all $\mathfrak m \ne \mathfrak m_2$, these $x$ and $y$ commute, then $[g,\Gamma_\alpha]=1$ implies 
$$
\forall \mathfrak m\in \mathfrak M\, \forall \beta\in \Gamma_\alpha \
[g_{\mathfrak m}, x_\beta (1_{R_{\mathfrak m}})]=e_{G_{\mathfrak m}}.
$$

Therefore if $g$ belongs to the centralizer of the set $\Gamma_\alpha$ in the whole Chevalley group $\widetilde G$, then each its component $g_{\mathfrak m}$ belongs to the centralizer of the corresponding set 
$$
\Gamma_{\alpha, \mathfrak m}=\{ x_\beta (1_{R_{\mathfrak m}})\mid [x_\beta(1_{R_{\mathfrak m}}),x_\alpha(1_{R_{\mathfrak m}})]=e\}.
$$
Since all $R_{\mathfrak m}$ are local rings, from the previous section we see that 
$$
\forall \mathfrak m\in \mathfrak M\ g_{\mathfrak m}=x_\alpha (t_{\mathfrak m}) \cdot C_{\mathfrak m},\text{ where }
t_{\mathfrak m}\in R_{\mathfrak m}\text{ and }C_{\mathfrak m}\in Z(G_\pi(\Phi, R_{\mathfrak m}))
$$
(or in the case $\mathbf C_l$, $l\geqslant 2$,
$$
\forall \mathfrak m\in \mathfrak M\ g_{\mathfrak m}=x_{e_1+e_2} (t_{\mathfrak m})x_{2e_1} (r_{\mathfrak m})x_{e_2} (s_{\mathfrak m}) \cdot C_{\mathfrak m},\text{ where }
t_{\mathfrak m}, r_{\mathfrak m}, s_{\mathfrak m}\in R_{\mathfrak m}\text{ and }C_{\mathfrak m}\in Z(G_\pi(\Phi, R_{\mathfrak m})).
$$

It means that 
$$
g=(x_\alpha (t_{\mathfrak m})\cdot C_{\mathfrak m})_{\mathfrak m\in \mathfrak M}=x_\alpha (t)\cdot C,\text{ where }t\in \widetilde R\text{ and } C\in Z(\widetilde G)
$$
(or in the case $\mathbf C_l$, $l\geqslant 2$,
$$
g=(x_{e_1+e_2} (t_{\mathfrak m})x_{2e_1} (r_{\mathfrak m})x_{e_2} (s_{\mathfrak m})\cdot C_{\mathfrak m})_{\mathfrak m\in \mathfrak M}=x_{e_1+e_2} (t)x_{e_1}(r)x_{2e_2}(s)\cdot C,\text{ where }t,r,s\in \widetilde R\text{ and } C\in Z(\widetilde G)
$$
The theorem is completely proved.
\end{proof}

\newpage

\section{Diophantine structure in large subgroups of Chevalley groups $G_\pi(\Phi,R)$}\leavevmode

In this section we show that many important subgroups of the Chevalley groups are Diophantine. We freely use notation from Preliminaries.

\subsection{One-parametric subgroups $X_\alpha$ are Diophantine in large subgroups of $G_\pi(\Phi,R)$}\leavevmode

We start with the following key result.

\begin{proposition}\label{theorM4.1}
Let $G$ be a large subgroup of $G_\pi(\Phi,R)$, where $\Phi$ is indecomposable root system of the rank $\ell > 1$, $\Phi\ne \mathbf C_2$, $R$ is an arbitrary commutative rings with~$1$. Then for any root $\alpha\in \Phi$  the subgroup $X_\alpha$ is Diophantine in~$G$ \emph{(}defined with constants  $\mathbf x=\{ x_\beta(1)\mid \beta \in \Phi\})$.
\end{proposition}

\begin{proof}
By the Theorem~\ref{double_centr_arbitrary} for all root systems except $\mathbf C_\ell$, $\ell\geqslant 2$, for any $\alpha\in \Phi$ we have
$$
C_G(\Gamma_\alpha)=X_\alpha\cdot Z(G),
$$
where $Z(G)$ is the center of the group~$G$.

In all root systems $\Phi\ne \mathbf C_\ell$ of the rank $> 1$ there exist two roots $\alpha$ and $\beta$ forming together the basis of the root system~$\mathbf A_2$. 

Therefore,
$$
X_{\alpha+\beta} = [C_G(\Gamma_\alpha), x_\beta(1)],
$$
so the subgroup $X_{\alpha+\beta}$ is Diophantine in~$G$.

\medskip

{\bf 1. Roots systems $\mathbf A_\ell, \mathbf D_\ell, \mathbf E_\ell, \mathbf F_4$}. Since all roots of the same length are conjugated up to action of the group~$W$, we proved now that for the root systems $\mathbf A_\ell, \mathbf D_\ell, \mathbf E_\ell, \mathbf F_4$ all subgroups $X_\alpha$, $\alpha\in \Phi$,  are Diophantine in~$G$; for the root systems $\mathbf G_2$ and $\mathbf B_\ell$, $\ell\geqslant 3$, the subgroups $X_\alpha$, $\alpha$ is long, are Diophantine in~$G$.

\medskip

{\bf 2. Roots system $\mathbf G_2$}. In the case $\Phi=\mathbf G_2$ all Chevalley groups are adjoint, therefore their centers are always trivial (see~\cite{Abe-Hurley}). So for $\mathbf G_2$ we have $C_G(\Gamma_\alpha)=X_\alpha$ for all $\alpha\in \Phi$.

\medskip

{\bf 3. Roots systems $\mathbf B_\ell, \ell\geqslant 3$}. For the case $\mathbf B_\ell$, $\ell \geqslant 3$, we already proved that all $X_\alpha$ for long roots $\alpha$ and all $X_\beta Z(G)$ for short roots~$\beta$ are Diophantine in~$G$. Let us take two roots $\alpha,\beta\in \Phi$, where $\alpha$ is long, $\beta$ is short and they form the system $\mathbf B_2$. 

Since 
$$
[x_\alpha(t),x_\beta(1)]=x_{\alpha+\beta}(\pm t)x_{\alpha+2\beta}(\pm t), \quad \alpha+\beta\text{ is short}, \alpha+2\beta\text{ is long},
$$
let $w\in W$ be such that $w(\alpha)=\alpha+2\beta$, let for the sake of certainty both signs are $+$,
then 
$$
X_{\alpha+\beta} Z(G) \cap [X_\alpha ,x_\beta(1)]\cdot X_{\alpha+2\beta}
$$
is precisely $X_{\alpha+\beta}$. Therefore $X_{\alpha+\beta}$ and then $X_\beta$ are Diophantine in~$G$ as intersection of two Diophantine sets.

\medskip

{\bf 4. Roots systems $\mathbf C_\ell, \ell\geqslant 3$}.
Now we only need to prove our statement for the root system $\Phi=\mathbf C_\ell$, $\ell\geqslant 3$.
In this system
$$
C_G(\Gamma_{e_1-e_2})=X_{e_1-e_2}X_{2e_1}X_{-2e_2}C,
$$
and
$$
[x_{e_1-e_2}(t)x_{2e_1}(r)x_{-2e_2}(s)c,x_{e_2-e_3}(1)]=x_{e_1-e_3}(t)x_{-e_2-e_3}(s)x_{-2e_3}(s),
$$
after that 
$$
[ x_{e_1-e_3}(t)x_{-e_2-e_3}(s)x_{-2e_3}(s),x_{e_1+e_2}(1)]=x_{e_1-e_3}(s),
$$
therefore
$$
[[C_G(\Gamma_{e_1-e_2}), x_{e_2-e_3}(1)],x_{e_1+e_2}(1)]=X_{e_1-e_3},
$$
therefore $X_\alpha$ for any short $\alpha\in \Phi$ is Diophantine in~$G$. 

Now we know that $X_\alpha \cdot Z(G)$ for any long $\alpha\in \Phi$ and $X_\beta$ for any short $\beta \in \Phi$ are Diophantine in~$G$.  Let us again take two roots $\alpha,\beta\in \Phi$, where $\alpha$ is long, $\beta$ is short and they form the system $\mathbf C_2=\mathbf B_2$. Then the set
$$
X_{\alpha+2\beta}=X_{\alpha+2\beta} Z(G) \cap [X_\alpha Z(G),x_\beta(1)]\cdot X_{\alpha+\beta}
$$
is Diophantine in~$G$ as the intersection of two Diophantine sets. 
\end{proof}

\begin{proposition}\label{theorM4.1-B2-new}
Let $G$ be a large subgroup of $G_\pi(\mathbf C_2,R)$, $\mathbf C_2=\{ \pm e_1\pm e_2, \pm 2e_1, \pm 2e_2\}$, where $R$ is an arbitrary commutative ring with~$1$ and either $\pi =\ad$ or $1/2\in R$. Then for every $\gamma\in \Phi$ the subgroup  $X_\gamma$  is Diophantine in~$G$. 
\end{proposition}

\begin{proof}
{\bf Case 1.} If our Chevalley group is adjoint ($\pi =\ad$), then its center is trivial.
 For any long root $\gamma\in \Phi$ we know that $X_\gamma=X_\gamma Z(G)=C_G(\Gamma_\gamma)$ is Diophantine in~$G$.
 
 Let 
 $$
 [x_{2e_1}(t), x_{e_2-e_1}(u)]=x_{e_1+e_2}(\pm tu)x_{2e_2}(\pm tu^2),
 $$
 for example
 $$
 [x_{2e_1}(t), x_{e_2-e_1}(u)]=x_{e_1+e_2}( tu)x_{2e_2}( tu^2).
 $$
 Suppose also that for $w\in W$ we have $w(2e_1)=2e_2$. Then 
 $$
 X_{e_1+e_2}=\{ [y,x_{e_2-e_1}(1)]\cdot w y^{-1}w^{-1}\mid y\in X_{2e_1}\}.
 $$
 Therefore $X_\gamma $ are Diophantine in~$G$ for all short roots~$\gamma\in \Phi$.

 \medskip

 {\bf Case 2.} Let us suppose that $1/2\in R$. As above we will use that 
for  any long root $\gamma\in \Phi$ the set  $X_\gamma Z(G)$ is Diophantine in~$G$.
According to the previous case it is evident that the set  $X_\gamma Z(G)$ is Diophantine in~$G$ also for any short root.
Then using the relation
$$
[x_{e_1+e_2}(t),x_{e_2-e_1}(u)]=x_{2e_2}(\pm 2tu).
$$
we have
$$
[X_{e_1+e_2}Z(G), x_{e_2-e_1}(1/2)]=X_{2e_2},
$$
therefore $X_\gamma$ is Diophantine in~$G$ for any long~$\gamma$, and as in the case~1 for any short root~$\gamma$.
\end{proof}

In the case $\mathbf C_2$, $\pi=sc$, $1/2\notin R$ we need the special auxiliary set $Y$ to be Diophantine.

\begin{proposition}\label{theorM4.1-B2}
Let $G$ be a large subgroup of $G_{sc}(\mathbf C_2,R)$,  where $R$ is an arbitrary commutative rings without~$1/2$. Then the subgroup  $Y_{e_1+e_2}=\{ x_{e_1+e_2}(t)x_{2e_2}(t)\mid t\in R\}$  is Diophantine in~$G$. 
\end{proposition}

\begin{proof}
In the case $\mathbf C_2$
we know that 
$$
C_G(\Gamma_{\pm 2e_i})=X_{\pm 2e_i} C,\quad  i=1,2,\ C=Z(G)
$$
and
$$
C_G(\Gamma_{\pm e_1\pm e_2})=X_{\pm e_1\pm e_2}X_{\pm 2e_1}X_{\pm 2e_2}C,\quad C=Z(G).
$$

Also we remember
\begin{align*}
[x_{2e_1}(t), x_{e_2-e_1}(u)]&=x_{e_1+e_2}(\pm tu)x_{2e_2}(\pm tu^2),\\
[x_{e_1+e_2}(t),x_{e_2-e_1}(u)]&=x_{2e_2}(\pm 2tu).
\end{align*}

Therefore 
$$
[C_G(\Gamma_{2e_1}),x_{e_2-e_1}(1)]=Y_{e_1+e_2}.
$$

\end{proof}

\subsection{E-interpretability of Chevalley groups}\leavevmode

\begin{theorem}\label{theorM6.1}
 Let $G$ be a large subgroup of a Chevalley group $G_\pi(\Phi,R)$,  where $\Phi$ is indecomposable root system of the rank $\ell > 1$, $R$ is an arbitrary commutative rings with~$1$. Then the ring $R$ is
e-interpretable in~$G$ $($using constants from the set $C_\Phi = \{ x_\alpha(1) \mid \alpha \in \Phi\})$ on every $X_\alpha$, $\alpha\in \Phi$, except the case $G=G_{sc}(\mathbf C_2,R)$, $1/2\notin R$, where $R$ is e-interpretable in~$G$ on the set $Y_\alpha$ from Proposition~\ref{theorM4.1-B2}.
\end{theorem}
\begin{proof}
 There are four cases to consider: 
 
 --- roots included in the system $\mathbf A_2$,
 
 --- short roots of the system~$\mathbf G_2$, 
 
 --- roots included in the system $\mathbf B_2/\mathbf C_2$, where all $X_\alpha$ are Diophantine, 
 
 --- and $\mathbf C_2$ with the sets~$Y_\alpha$.

 \medskip

 {\bf Case 1.} Suppose that we have a root system~$\mathbf A_2=\langle \alpha,\beta\rangle$ and we want to interpret the ring~$R$ on~$X_{\alpha+\beta}$. We will turn the set $X_{\alpha+\beta}$ into a ring $\langle X_{\alpha+\beta},\oplus, \otimes\rangle$
as follows.

For  $x,y\in X_{\alpha+\beta}$ we define
$$
x\oplus y=x\cdot y.
$$
Note that if $x=x_{\alpha+\beta}(a)$, $y=x_{\alpha+\beta}(b)$, then $xy=x_{\alpha+\beta}(a+b)$, which corresponds to the addition in~$R$.

To define $x\otimes y$ for given $x,y\in X_{\alpha+\beta}$  we need some notation. Let $x_1,y_1\in G$ be
such that 
$$
x_1\in X_\alpha\text{ and }[x_1, x_\beta(1)] = x;\quad  y_1\in X_\beta\text{ and }[x_\alpha(1),y_1] = y.
$$
Note that such $x_1,y_1$ always exist and unique,  if $x=x_{\alpha+\beta}(a)$, $y=x_{\alpha+\beta}(b)$, then
$x_1=x_\alpha(a)$,  $y_1=x_\beta(b)$. Now define
$$
x\otimes y:= [x_1,y_1].
$$
Observe, that in this case
$$
[x_1, y_1]= [x_\alpha(a),x_\beta(b)]=x_{\alpha+\beta}(ab).
$$
so corresponds to the multiplication in~$R$. To finish the proof we need two claims.

{\bf Claim 1.} \emph{The map $a\mapsto x_{\alpha+\beta}(a)$ gives rise to a ring isomorphism $R\to 
\langle X_{\alpha+\beta},\oplus, \otimes\rangle$.}

This is clear from the argument above.

{\bf Claim 2}. \emph{The ring $ \langle X_{\alpha+\beta},\oplus, \otimes\rangle$
 is e-interpretable in~$G$.}

To see this, observe first that, as was mentioned above, $X_{\alpha+\beta}$ is Diophantine in~$G$.
The defined addition is clearly Diophantine in~$G$. Since the subgroups
$X_\alpha$ and $X_\beta$ are Diophantine in~$G$ the multiplication $\otimes$ is also Diophantine in~$G$. This proves the case~1.

\medskip

{\bf Case 2.} Suppose that we want to interpret a ring~$R$ on some short root of the system~$\mathbf G_2$. Since long roots of~$\mathbf G_2$ form the subsystem~$\mathbf A_2$, then the ring $R$ is already interpreted on all $X_\alpha$ for long roots~$\alpha$ and it is sufficient to find a Diophantine isomorphism $\mu: X_\alpha\to X_{\alpha+\beta}$, where $\alpha$ is long and $\beta$ is short.

We will use the relation
$$
[x_\alpha(t),x_\beta(u)]=x_{\alpha+\beta}(tu)x_{\alpha+3\beta}(-tu^3)x_{\alpha+2\beta}(-tu^2)x_{2\alpha+3\beta}(t^2u^3),
$$
that shows
$$
x_{\alpha+\beta}(t)=\mu (x_\alpha(t)) = X_{\alpha+\beta}\cap [x_{\alpha}(t),x_\beta(1)]X_{2\alpha+3\beta}X_{\alpha+3\beta}X_{\alpha+2\beta}.
$$
Since $\mu$ is Diophantine, and $R$  is e-interpretable in~$G$ on~$X_\alpha$, then  $R$  is e-interpretable in~$G$ also on~$X_{\alpha+\beta}$ and therefore on all $X_\gamma$ for all roots.

\medskip

{\bf Case 3.} Now let us consider the root system $\mathbf C_2$, where all $X_\gamma$, $\gamma \in \Phi$, are Diophantine. 

Of course, for  $x,y\in X_{\gamma}$ we define
$$
x\oplus y=x\cdot y
$$
and it is equivalent to addition in~$R$ for any root $\gamma\in \Phi$.
\medskip

To define $x\otimes y$ for given $x,y\in X_\gamma$ we will start with defining an isomorphism~$\mu$ between $X_\alpha$ and $X_{\alpha+\beta}$, where $\alpha$ is long, $\beta$ and $\alpha+\beta$ are short, with $\mu (x_\alpha (t))=x_{\alpha+\beta}(t)$.

This isomorphism $\mu$ can be determined for example by
$$
x_{\alpha+\beta}(t)=\mu(x_\alpha(t))=X_{\alpha+\beta} \cap [x_\alpha(t),x_\beta(1)]\cdot X_{\alpha+2\beta}.
$$

Now when $\mu$ is defined it is evidently sufficient to define the operation $\otimes$ only for one $X_\gamma$. We will do it for $\gamma=\alpha+\beta=e_1+e_2$.

Let $x,y\in X_{\alpha+\beta}$ and $x_1\in X_\alpha$ be defined as $\mu^{-1} (x)$, $y_2\in X_\beta$ be $y^{w_{e_1}^{-1}}$. Then we will define $x\otimes y$ as
$$
x\otimes y=X_{\alpha+\beta}\cap [x_1,y_2]\cdot X_{\alpha+2\beta}.
$$
If $x=x_{\alpha+\beta}(a)$, then $x_1=x_\alpha(a)$; if $y=x_{\alpha+\beta}(b)$, then $y_2=x_\beta(b)$. In this case $[x_1,y_2]=x_{\alpha+\beta}(ab)\cdot x_{\alpha+2\beta}(\pm ab^2)$ and 
$$
X_{\alpha+\beta}\cap x_{\alpha+\beta}(ab)\cdot x_{\alpha+2\beta}(\pm ab^2)\cdot X_{\alpha+2\beta}=x_{\alpha+\beta}(ab).
$$
Therefore the  ring $ \langle X_{\alpha+\beta},\oplus, \otimes\rangle\cong R$
 is e-interpretable in~$G$.

\medskip

{\bf Case 4.} 
If $\Phi= \mathbf C_2$ and the sets $X_\alpha$ are not Diophantine in~$G$, then by Proposition~\ref{theorM4.1-B2} the subgroup $Y_{e_1+e_2}=\{ x_{e_1+e_2}(t)x_{2e_2}(t)\mid t\in R\}$  is Diophantine in~$G$.

We e-interpret $R$ on $Y_{e_1+e_2}$ turning it into a ring $\langle Y_{e_1+e_2},\oplus, \otimes\rangle$
as follows.

For  $x,y\in Y_{e_1+e_2}$ we define
$$
x\oplus y=x\cdot y.
$$
Note that if $x=x_{e_1+e_2}(a)x_{2e_2}(a)$, $y=x_{e_1+e_2}(b)x_{2e_2}(b)$, then 
$$
xy=x_{e_1+e_2}(a)x_{2e_2}(a)x_{e_1+e_2}(b)x_{2e_2}(b)=(x_{e_1+e_2}(a)x_{e_1+e_2}(b))(x_{2e_2}(a)x_{2e_2}(b))=x_{e_1+e_2}(a+b)x_{2e_2}(a+b),
$$
 which corresponds to the addition in~$R$.

To define $x\otimes y$ for given $x,y\in Y_{e_1+e_2}$  we need to use several tricks. 

First, let $x_1\in G$ be
such that 
$$
x_1\in X_{2e_1}C=C(\Gamma_{2e_1})\text{ and }[x_1, x_{e_2-e_1}(1)] = x.
$$
Note that if $x=x_{e_1+e_2}(a)x_{2e_2}(a)$, then necessarily $x_1=x_{2e_1}(a) c$, $c\in Z(G)$.

Second, let $y_2$ be 
$$
y^{w_{2e_1}(1)}=x_{\mathbf w_{2e_1}(e_1+e_2)}(b)x_{\mathbf w_{2e_1}(2e_2)}(b)=x_{e_2-e_1}(b)x_{2e_2}(b).
$$
We see that
$$
[x_1,y_2]=[x_{2e_1}(a),x_{e_2-e_1}(b)x_{2e_2}(b)]=[x_{2e_1}(a),x_{e_2-e_1}(b)]=x_{e_2+e_1}(ab)x_{2e_2}(ab^2).
$$
If we set 
$$
x\otimes y:= ( [x_1,y_1]C(\Gamma_{2e_2})\cap Y_{e_1+e_2})=x_{e_1+e_2}(ab)x_{2e_2}(ab),
$$
since $[x_1,y_1]C(\Gamma_{2e_2})=x_{e_2+e_1}(ab)x_{2e_2}(t)c$, $t\in R$, $c\in Z(G)$.

Therefore the case 4 is also complete.
\end{proof}

Now we prove the converse of Theorem~\ref{theorM6.1}. The result, we believe, is known in folklore.

\begin{proposition}\label{prop6.4}
All Chevalley groups $G_\pi(\Phi,R)$, where $R$ is an arbitrary commutative ring are all e-interpretable in~$R$ (not using constants from $R$ other then integers). 
\end{proposition}

\begin{proof}
 We represent an $n\times n$-matrix $x = (x_{ij})$ with entries in~$R$ by an $n^2$-tuple $\overline x$ over~$R$, where
$$
\overline x = (x_{11},\dots, x_{1n}, x_{21},\dots, x_{n1},\dots , x_{nn}).
$$
The matrix multiplication $\otimes $ on tuples from~$R^{n^2}$
is defined by
$$
\overline x \otimes \overline y = \overline z \Longleftrightarrow \bigwedge_{i,j=1}^n z_{ij}=P_{ij}(\overline x,\overline y),
$$
where $P_{ij}(\overline x;\overline y)$ is integer polynomial $\sum\limits_{k=1}^n x_{ik}y_{kj}$. The multiplication $\otimes$ is clearly Diophantine. To finish the description of the interpretations of the groups $G_\pi (\Phi,R)$ in~$R$ it suffices to define the corresponding subsets of $R^{n^2}$ by Diophantine formulas. 

But it is so by definition of Chevalley groups, which are all defined by finite system of polynomial equations with integer coefficients. 
\end{proof}

We are not able to show that the elementary Chevalley group $E_\pi(\Phi,R)$ is e-interpretable in~$R$ for any commutative ring~$R$. 

However, the following holds.

\begin{theorem}\label{Theorem6.6}
 If an elementary Chevalley group $E_\pi(\Phi,R)$ has bounded elementary generation, then $E_\pi(\Phi,R)$ is $e$-interpretable in~$R$.
\end{theorem}

\section{Diophantine problem in Chevalley groups}\leavevmode

In this section we  study  Diophantine problem  in Chevalley groups over rings.   Our  arguments are often similar to the corresponding ones   for classical linear groups from \cite{Myasnikov-Sohrabi2}, so we either state the results without proofs and refer the reader  to \cite{Myasnikov-Sohrabi2} or give a short sketch of the proof.

\subsection{General reductions}

We consider here the Diophantine problems of the type $\mathcal D_C(G_\pi(\Phi,R))$, where
$C$ is a countable subset of  $G_\pi(\Phi,R)$ equipped with an enumeration $\nu: \mathbb N \to C$.
Denote by~$R_C$ the set of all elements of~$R$ that occur in matrices from~$C$.
The enumeration $\nu$ gives rise to an enumeration $\mu: \mathbb N\to R_C$, where to construct~$\mu$ it suffices to enumerate matrices in~$C$ with respect to~$\nu$, for each matrix $\nu(n)$ enumerate its entries in some fixed
order, and combine all these into an enumeration~$\mu$.

Now we can prove the main result of the paper.

\begin{theorem}\label{Theorem6.5}
 If $\Phi$ is an indecomposable root system of a rank $> 1$, $R$ is an arbitrary commutative ring with~$1$, then  the Diophantine problem in any Chevalley group
$G_\pi(\Phi,R)$ is Karp equivalent to
the Diophantine problem in~$R$. More precisely:
\begin{itemize}
\item [1)] If $C$ is a countable subset of $G_\pi(\Phi,R)$  then $\mathcal D_C(G_\pi(\Phi,R))$ Karp reduces to $\mathcal D_{R_C} (R)$.
\item [2)] If $T$ is a countable subset of $R$ then there is a countable subset $C_T$ of $G_\pi(\Phi,R)$  such that $\mathcal D_{T} (R)$ Karp reduces to $\mathcal D_{C_T}(G_\pi(\Phi,R))$.
\end{itemize}
\end{theorem}

\begin{proof}
1) follows directly from Proposition \ref{prop6.4} (see Lemma 7.1 from \cite{Myasnikov-Sohrabi2}). 2) comes from Theorem \ref {theorM6.1}.
\end{proof}

\begin{theorem}\label{Theorem6.8}
 Let $\Phi$ be an indecomposable root system of a rank $> 1$ and  $R$ an arbitrary commutative ring with~$1$. If the elementary Chevalley group $E_\pi(\Phi,R)$ has bounded elementary generation, then the Diophantine problem in 
$E_\pi(\Phi,R)$ is Karp equivalent to
the Diophantine problem in~$R$. 
\end{theorem}
\begin{proof}
The result follows from Theorem \ref{Theorem6.6}.
\end{proof}

\subsection{Diophantine problem in Chevalley groups over rings of algebraic integers and number fields}
By a \emph{number field} $F$ we mean a finite algebraic extension of~$\mathbb Q$. The \emph{ring of
algebraic integers}~$\mathcal{O}_F$ of a number field~$F$ is the subring of~$F$ consisting of all roots of
monic polynomials with integer coefficients.

It is a classical result that the Diophantine problem in~$\mathbb Z$ is undecidable~\cite{M52}.

\begin{theorem}[compare with Theorem 7.2 from \cite{Myasnikov-Sohrabi2}]\label{Theor7.2} 
If $\Phi$ is a indecomposable root system of a rank $>1$, then the Diophantine problem in all Chevalley groups $G_\pi(\Phi,\mathbb Z)$ is Karp equivalent to the Diophantine problem in~$\mathbb Z$, in particular, it is
undecidable.
\end{theorem}

The following is one of the major conjectures in number theory.

\begin{conjecture} \label{con:majorNT}
The Diophantine problem in $\mathbb{Q}$, as well as in any number field~$F$, or any ring of algebraic integers $\mathcal{O}$, is undecidable.
\end{conjecture}

For $\mathbb{Q}$ and any its finite extension $F$ the conjecture above is wide open. However, for the rings of algebraic integers $\mathcal{O}_F$ of the fields $F$ there are results where the undecidability of the Diophantine problem is confirmed. Namely,  it is known that $\mathbb{Z}$  is Diophantine in $\mathcal{O}_F$ if $[F : \mathbb{Q}] = 2$ 
or $F$  is totally real \cite{Denef1,Denef_3}, or $[F : \mathbb{Q}] > 3$  and $F$  has
exactly two nonreal embeddings into the field of complex numbers \cite{Pheidas_1988}, or $F$ is an Abelian number field \cite{Sha_Shla}.  We refer to two surveys  and a book \cite{M59,Poonen,M75} for details on this matter.

The following result moves the Diophantine problem in Chevalley groups over number fields or rings of algebraic integers from group theory to number theory. 

\begin{theorem}\label{Theorem7.2_2}
Let $\Phi$ be  an indecomposable root system of a rank $>1$ and $R$ either a number field or a ring of algebraic integers. Then Conjecture \ref{con:majorNT} holds for $R$ if and only if  the Diophantine problem in the Chevalley group $G_\pi(\Phi,R)$ is undecidable.
\end{theorem}

\subsection{Diophantine problem in Chevalley groups over finitely generated commutative rings}

To move forward we need to recall some definitions.

The \emph{characteristic} of a ring with multiplicative identity (i.e.\ a \emph{unitary} ring) is the minimum positive integer $n$ such that $1 + \overset{n}{\dots} + 1 = 0$. By \emph{rank} of a ring $R$ we refer to the rank of $R$ seen as an abelian group (i.e.\ forgetting its multiplication operation): that is, the maximum number $m$ of nonzero elements $r_1, \dots, r_m\in R$ such that whenever $a_1r_1 + \dots + a_m r_m=0$ for some integers $r_1, \dots, r_m$, we have $r_ia_i=0$ for all $i=1, \dots, m$. If $R$ is an integral domain, then its rank coincides with its dimension as a $\mathbb{F}_p$-vector space if $R$ has positive characteristic $p$, and otherwise it coincides with the dimension of $R$ seen as a $\mathbb{Z}$-module.

The following result from \cite{M34} describes the current state of the Diophantine problem in finitely generated commutative rings. Note that in \cite{Denef1,Denef2}  Denef showed that the Diophantine problems in polynomial rings with coefficients in integral domains are undecidable.

\begin{theorem}[\cite{M34}] \label{th:9}

 Let $R$ be an infinite finitely generated associative commutative unitary ring.
Then  one of the following holds:

\begin{enumerate}
    \item If $R$ has positive characteristic $n> 0$, then the ring of polynomials  $\mathbb{F}_p[t]$ is e-interpretable in $R$ for some transcendental element $t$ and some prime integer $p$; and $\mathcal{D}(R)$ is undecidable.
    \item  If  $R$ has zero characteristic and it has infinite rank then the same conclusions as above hold:    the ring of polynomials  $\mathbb{F}_p[t]$ is e-interpretable in $R$ for some $t$ and $p$; and $\mathcal{D}(R)$ is undecidable. 
    \item  If  $R$ has zero characteristic and it has finite rank then a ring of algebraic integers $\mathcal{O}$ is e-interpretable in $R$. 
\end{enumerate}
\end{theorem}

This together with Theorem \ref{Theorem6.5} implies the following result which completely clarifies the situation with the Diophantine problem in Chevalley groups over infinite finitely generated commutative unitary rings. 

\begin{theorem}
Let $\Phi$ be an indecomposable root system of a rank $> 1$, $R$ is an arbitrary infinite finitely generated commutative ring with~$1$,  and 
$G_\pi(\Phi,R)$  the corresponding  Chevalley group. Then:

\begin{itemize}
\item [1)]  If $R$ has positive characteristic then the Diophantine problem in $G_\pi(\Phi,R)$ is undecidable.
\item[2)] If  $R$ has zero characteristic and it has infinite rank then the Diophantine problem in $G_\pi(\Phi,R)$ is undecidable. 
\item [3)] If  $R$ has zero characteristic and it has finite rank then the Diophantine problem in some ring of algebraic integers $\mathcal{O}$ is Karp reducible to the Diophantine problem in $G_\pi(\Phi,R)$. Hence if Conjecture \ref{con:majorNT} holds then the Diophantine problem in  $G_\pi(\Phi,R)$ is undecidable.

\end{itemize}

\end{theorem}

\subsection{Diophantine problem in Chevalley groups over algebraically closed fields}

Let $R$ be an algebraically closed field.
We need the following known results about $R$ (see \cite{Myasnikov-Sohrabi2} for details and references).

\begin{itemize}
    \item [1)] If  $A$ is a computable subfield of $R$  then the first-order theory $Th_A(R)$ of $R$ with constants from $A$ in the language  is decidable. In particular, the Diophantine problem $\mathcal{D}_A(R)$  is decidable.  
    \item [2)]  If  $A$ is a computable subfield of $R$ then   the algebraic closure $\bar A$ of $A$ in $R$ is  computable. 

\end{itemize}

\begin{theorem}\label{Theorem-alg-closed}
Let $\Phi$ be an indecomposable root system of a rank $> 1$, $R$ an algebraically closed field, and $G_\pi(\Phi,R)$ the corresponding Chevalley group. If $A$ is a computable subfield of $R$, then the Diophantine problem in $G_\pi(\Phi,R)$ with constants from $G_\pi(\Phi,A)$ is decidable (under a proper enumeration of $G_\pi(\Phi,A)$).

\end{theorem}

\subsection{Diophantine problem in Chevalley groups over reals}

Let  $R = \mathbb{R}$ be  the field of real numbers and $A$  a countable (or finite) subset of $\mathbb{R}$. 

Our treatment of the Diophantine problem in Chevalley groups over $\mathbb{R}$ is based on the following two results on  the Diophantine problem in $\mathbb R$ which are known in the folklore. For details we refer to \cite{Myasnikov-Sohrabi2}.

\begin{proposition}[\cite{Myasnikov-Sohrabi2}, Proposition 7.4] \label{pr:dec-ord} 
Let $A$ be a finite or countable subset of $\mathbb{R}$.  Then the Diophantine problem in $\mathbb{R}$ with coefficients in $A$ is decidable if and only if the ordered subfield $F(A)$ is computable. Furthermore, in this case the whole first-order theory $Th_A(\mathbb{R})$ is decidable.  
\end{proposition}

Recall that a real $a \in \mathbb{R}$ is \emph{computable} if its standard decimal expansion $a = a_0.a_1a_2 \ldots$ is computable, i.e., the integer function $n \mapsto a_n$ is computable. In other words, $a$ is computable if and only if one can effectively approximate it by rationals with any precision. The set of all computable reals $\mathbb{R}^c$ forms a real closed subfield of $\mathbb{R}$, in particular $\mathbb{R}^c$  is first-order equivalent to $\mathbb{R}$.

In the following Proposition we collect some facts about computable ordered subfields of $\mathbb{R}$. 

\begin{proposition}[\cite{Myasnikov-Sohrabi2}, Proposition 7.5] \label{pr:main-R} 
The following holds:
\begin{itemize}
    \item [1)]  Every ordered computable subfield of $\mathbb{R}$ is contained in $\mathbb{R}^c$.
    
    \item [2)]  The ordered subfield $\mathbb{R}^c \leq \mathbb{R}$ with the induced order from $\mathbb{R}$ is not computable.
    \item [3)]    If $F$ is a computable ordered field, then its real closure is also computable. In particular, if $F$ is a computable subfield of $\mathbb{R}$ then the  algebraic closure  $\bar F$ of $F$ in $\mathbb{R}$ is a computable ordered field. 
    \item [4)] If $a_1, \ldots,a_m$ are computable reals then the ordered subfield $\mathbb{Q}(a_1, \ldots,a_m) \leq \mathbb{R}$ with the induced order from $\mathbb{R}$ is computable.
    
\end{itemize}
\end{proposition}

\begin{corollary} \label{co:dec-R} The following holds:
\begin{itemize}
    \item The Diophantine problem in $\mathbb{R}$ with coefficients in  $\mathbb{R}^c$ is undecidable;
    \item The  Diophantine problem in $\mathbb{R}$ with coefficients in any finite subset of  $\mathbb{R}^c$ is  decidable;
    \item The  Diophantine problem in $\mathbb{R}$ with coefficients in $\{a\}$, where $a$ is not computable, is undecidable.
\end{itemize}
\end{corollary}

Recall that a matrix $A \in \GL_n(\mathbb{R})$ is called \emph{computable} if all entries in $A$ are computable real numbers. 

Chevalley groups $G_\pi(\Phi,\mathbb{R})$ are matrix algebraic groups over $\mathbb{R}$, hence one can view their elements as matrices 

\begin{theorem} \label{th:main-R} 
Let $\Phi$ be an indecomposable root system of a rank $> 1$ and $G_\pi(\Phi,\mathbb{R})$ the Chevalley group over the field of real numbers $\mathbb{R}$. If $A$ is a computable ordered subfield of $\mathbb R$ then   the first-order theory  $Th(G_\pi(\Phi,\mathbb{R}))$ with constants from $G_\pi(\Phi,A)$ is decidable. In particular, 
the Diophantine problem in $G_\pi(\Phi,\mathbb{R})$ with constants from $G_\pi(\Phi,A)$ is decidable (under a proper enumeration of $G_\pi(\Phi,A)$).
 \end{theorem}

 \begin{theorem}\label{th:main-R-dop} 
Let $\Phi$ be an indecomposable root system of a rank $> 1$ and $G_\pi(\Phi,\mathbb{R}^c)$ the Chevalley group over the field of computable real numbers $\mathbb{R}^c$.  Then the  following holds:
\begin{itemize} 
\item [1)] 
 The Diophantine problem in the computable group $G_\pi(\Phi,\mathbb{R}^c)$  is undecidable. 
 \item [2)] For any finitely generated subgroup $C$ of $G_\pi(\Phi,\mathbb{R}^c)$  the Diophantine problem in $G_\pi(\Phi,\mathbb{R}^c)$ with coefficients in $C$ is decidable. 
 \end{itemize}
 
 \end{theorem}

We say that a matrix $A \in \GL_n(\mathbb{R})$ is \emph{computable} if all entries in $A$ are computable reals, i.e., $A \in \GL_n(\mathbb{R}^c)$. Hence the computable matrices in $\SL_n(\mathbb{R})$ are precisely the matrices from $\SL_n(\mathbb{R}^c)$. Since elements of a Chevalley group 
$G_\pi(\Phi,\mathbb R)$ are represented by matrices from $\SL_n(\mathbb{R})$ we say that an element $g \in G_\pi(\Phi,\mathbb{R})$ is computable if it is represented by a computable matrix from $\SL_n(\mathbb{R}^c)$.

 \begin{theorem} \label{th:incomp-R}
  Let $\Phi$ be an indecomposable root system of a rank $> 1$ and $G_\pi(\Phi,\mathbb{R})$ is the corresponding Chevalley group over the field of reals $\mathbb{R}$. If an element  $g \in E_\pi(\Phi,\mathbb{R})$  is not computable then the  Diophantine problem for equations with coefficients in $\{x_\alpha(1) \mid \alpha \in \Phi\} \cup \{g\}$ is undecidable in any large subgroup of $G_\pi(\Phi,\mathbb{R})$.

\end{theorem}

\subsection{Diophantine problem in Chevalley groups over $p$-adic numbers}

Similar to the case of reals  one can define computable $p$-adic numbers for every fixed prime $p$. Recall, that every $p$-adic number $a \in \mathbb{Q}_p$ has a unique presentation in the form $a = p^m\xi$, where $m \in \mathbb{Z}$ and $\xi$ is a unit in the ring $\mathbb{Z}_p$. In its turn, the unit $\xi$ is uniquely determined by a sequence of natural numbers $\{\xi(i)\}_{i \in \mathbb{N}}$, where 
$$
0\leq \xi(i) < p^{i+1}, \ \xi(i+1) = \xi(i) (\text{ mod } p^{i+1}), \ (i \in \mathbb{N}).
$$
The $p$-adic number $a = p^m\xi$ is computable if the sequence $i \to \xi(i)$ is computable. In this case the sequence $\{\xi(i)\}_{i \in \mathbb{N}}$ gives  an effective $p$-adic approximation of $\xi$. It is known (see, for example \cite{MR}), that the set $\mathbb{Q}_p^c$ of all computable $p$-adic numbers forms a subfield of $\mathbb{Q}_p$, such that $\mathbb{Q}_p \equiv \mathbb{Q}_p^c$.  Observe also that the ring $\mathbb{Z}_p$ is Diophantine in $\mathbb{Q}_p$. More precisely, if $p \neq 2$, then $\mathbb{Z}_p$ is defined in  $\mathbb{Q}_p$ by formula $\exists y (1+px^2 = y^2)$, while if $p= 2$ then $\mathbb{Z}_p$ is defined  by the formula $\exists y(1+2x^3 = y^3)$ (see \cite{M27}).

The following results were shown in \cite{MR}:
\begin{itemize} \label{th:fields}

    \item [a)]  $Th(\mathbb{Z}_p,a_1, \ldots,a_n)$ is decidable if and only if each of  $a_1, \ldots,a_n$ is a computable $p$-adic number.  
    \item [b)] $Th(\mathbb{Q}_p,a_1, \ldots,a_n)$ is decidable if and only if each of  $a_1, \ldots,a_n$ is a computable p-adic  number. 
      \item [c)] If   a $p$-adic integer $a$ is not computable then equations with constants from $\mathbb{Q} \cup \{a\}$ are undecidable in $\mathbb{Z}_p$. 
    \item [d)] If   a $p$-adic number $a \in \mathbb{Q}_p$ is not computable then equations with constants from $\mathbb{Q} \cup \{a\}$ are undecidable in $\mathbb{Q}_p$. 
\end{itemize}

\begin{theorem}\label{theor-padic-1}
Let $\Phi$ be an indecomposable root system of a rank $> 1$. Then the  following holds:
 \begin{itemize}
     \item [1)] Let $a_1, \ldots, a_m \in \mathbb{Q}_p^c$ and   $A=\mathbb{Q}(a_1, \ldots,a_m)$  is the subfield of $\mathbb{Q}_p$  generated by $a_1, \ldots, a_m$. Then    the first-order theory  $Th(G_\pi(\Phi,\mathbb{Q}_p))$ with constants from $G_\pi(\Phi,A)$ is decidable. In particular, 
the Diophantine problem in $G_\pi(\Phi,\mathbb{Q}_p)$ with constants from $G_\pi(\Phi,A)$ is decidable (under a proper enumeration of $G_\pi(\Phi,A)$).
     \item [2)] Let $a_1, \ldots, a_m \in \mathbb{Z}_p^c$ and   $A=\mathbb{Z}(a_1, \ldots,a_m)$  is the subring of $\mathbb{Z}_p$  generated by $a_1, \ldots, a_m$. Then    the first-order theory  $Th(G_\pi(\Phi,\mathbb{Q}_p))$ with constants from $G_\pi(\Phi,A)$ is decidable. In particular, 
the Diophantine problem in $G_\pi(\Phi,\mathbb{Q}_p)$ with constants from $G_\pi(\Phi,A)$ is decidable (under a proper enumeration of $G_\pi(\Phi,A)$).
 \end{itemize}
  \end{theorem}
  \begin{proof}
  It follows from Theorem \ref{Theorem6.5} and the results a) and b) above. 
  \end{proof}

We say that a matrix $A \in \GL_n(\mathbb{Q}_p)$ is \emph{computable} if all entries in $A$ are computable p-adic numbers, i.e., $A \in \GL_n(\mathbb{Q}_p^c)$. Hence the computable matrices in $\SL_n(\mathbb{Q}_p)$ are precisely the matrices from $\SL_n(\mathbb{Q}_p^c)$. Since elements of a Chevalley group 
$G_\pi(\Phi,\mathbb{Q}_p)$ are represented by matrices from $\SL_n(\mathbb{Q}_p)$ we say that an element $g \in G_\pi(\Phi,\mathbb{Q}_p)$ is computable if it is represented by a computable matrix from $\SL_n(\mathbb{Q}_p^c)$. Similarly, we define computable elements in a Chevalley group $G_\pi(\Phi,\mathbb{Z}_p)$.

  \begin{theorem}\label{theor-padic-2}
  Let $\Phi$ be an indecomposable root system of a rank $> 1$ and $G_\pi(\Phi,\mathbb{Q}_p)$ $(G_\pi(\Phi,\mathbb{Z}_p), p \neq 2)$ the corresponding Chevalley group over $\mathbb{Q}_p$ $(\mathbb{Z}_p)$. If an element  $g \in E_\pi(\Phi,\mathbb{Q}_p)$ $(g \in E_\pi(\Phi,\mathbb{Z}_p), p \neq 2)$ is not computable then the  Diophantine problem for equations with coefficients in $\{x_\alpha(1) \mid \alpha \in \Phi\} \cup \{g\}$ is undecidable in any large subgroup of $G_\pi(\Phi,\mathbb{Q}_p)$  $(G_\pi(\Phi,\mathbb{Z}_p), p \neq 2)$.  

\end{theorem}

\end{document}